\documentclass{amsart}

\usepackage{amsmath}
\usepackage{amssymb}
\usepackage{mathrsfs}
\usepackage{dsfont}
\allowdisplaybreaks
\newcommand{\A}{\mathcal{A}}
\newcommand{\C}{\mathbb{C}}
\newcommand{\E}{\mathbb{E}}
\newcommand{\R}{\mathbb{R}}
\newcommand{\N}{\mathbb{N}}
\newcommand{\Z}{\mathbb{Z}}

\let\epsilon\varepsilon

\providecommand{\norm}[1]{\lvert\lvert #1 \rvert\rvert}

\newtheorem{theorem}{Theorem}[section]
\newtheorem{lemma}[theorem]{Lemma}

\newtheorem{prop}[theorem]{Proposition}
\newtheorem{cor}[theorem]{Corollary}

\theoremstyle{definition}
\newtheorem{definition}[theorem]{Definition}
\newtheorem{example}[theorem]{Example}

\theoremstyle{remark}
\newtheorem{remark}[theorem]{Remark}

\numberwithin{equation}{section}

\newcommand{\abs}[1]{\lvert#1\rvert}
\newcommand{\floor}[1]{\lfloor#1\rfloor}
\newcommand{\ceil}[1]{\lceil#1\rceil}

\usepackage[shortlabels]{enumitem}

\begin{document}

\title[]{Lyapunov exponents of the spectral cocycle for topological factors of bijective substitutions on two letters}

\author{Juan Marshall-Maldonado}
\address{Juan Marshall-Maldonado, Department of Mathematics,
	Bar-Ilan University, Ramat-Gan, Israel}
\email{jgmarshall21@gmail.com}




\keywords{Spectrum, substitutions, exponential sums, Lyapunov exponents, symbolic dynamics, twisted cocycle}

\begin{abstract}
The present paper explores the spectral cocycle defined in \cite{bufetov2020spectral} in the special case of bijective substitutions on two letters, the most prominent example being the Thue-Morse substitution. We derive an explicit subexponential behaviour of the deviations from the expected exponential behavior. Moreover, these sharp bounds will be exploited to prove that the top Lyapunov exponent is greater or equal to the top exponent of the subshift topological factors after a renormalization. In order to obtain such results for the substitutive subshift factors, we define a special kind of sum, which is a multiple version of the twisted Birkhoff sum. For the particular case of the Thue-Morse substitution, we derive that the exponent is zero, we give an explicit subexponential bound for the twisted Birkhoff sums and we do the same for subshift topological factors.
\end{abstract}
\maketitle

\section{Introduction}
This paper continues the work of \cite{bufetov2014modulus,bufetov2020spectral} around the spectral cocycle defined explicitly in \cite{bufetov2020spectral}. The study of this extension of the Rauzy-Veech cocycle has given fruitful results as effective rates of weakly mixing for several classes of systems such as interval exchange transformations \cite{avila2023quantitative}, linear flows on translations surfaces of genera higher or equal than 2 \cite{forni2022twisted} and special suspension flows over more general S-adic systems (including substitutions) \cite{bufetov2018holder,bufetov2020spectral,JEP_2021__8__279_0}. Other consequences are the nature of the spectral measures involved, specifically, sufficient and necessary conditions for purely singular spectrum \cite{berlinkov2019singular,bufetov2020spectral,bufetov2022substitution, baake2019renormalisation}.\\

The results of this work are first steps towards a general understanding of the behavior of the top Lyapunov exponent (and more generally, the whole Lyapunov spectrum) of the spectral cocycle under conjugacy and semiconjugacy, which is a completely open question: subshift topological factors may be associated to more complex substitutions, in particular the cocycle matrices may be of higher dimension. Therefore, it is not clear that the estimates on the matrices are transfered to the factors: consider the next example (from \cite{dekking2014structure}). The Thue-Morse substitution is given by $\zeta_{TM}: 0 \mapsto 01, 1\mapsto 10$. A conjugate substitution (see subsection \ref{conjugacy} for background) $\sigma:\A \longrightarrow \A^+$ on the alphabet $\A=\{w_1,w_2,\dots,w_{12}\}$ is given by
\begin{align*}
&\sigma(w_1) = w_4w_{10}& &\sigma(w_2) = w_4w_{10}& &\sigma(w_3) = w_5w_{11}& &\sigma(w_4) = w_5w_{11}\\
&\sigma(w_5) = w_6w_{12}& &\sigma(w_6) = w_6w_{12}& &\sigma(w_7) = w_7w_1& &\sigma(w_8) = w_7w_1\\
&\sigma(w_9) = w_8w_2& &\sigma(w_{10}) = w_8w_2& &\sigma(w_{11}) = w_9w_3& &\sigma(w_{12}) = w_9w_3.
\end{align*}
Consider the subshift $X_\sigma\subset\A^\Z$ associated to $\sigma$ (see subsection \ref{background}), and $\delta$ the Kronecker delta. Our results show in particular that all the sums of the form
\[
\sum_{k=0}^{2^{n}-1} \delta_{x_k,a}\, e^{2\pi i k\omega}, \quad a\in\A
\] 
have a subexponential behavior for any $\mathbf{x} = (x_n)_{n\in\Z}\in X_\sigma$ and almost every $\omega\in[0,1)$. This is just one example (of infinitely many) constructed with the classical method introduced by M. Queff\'elec (see \cite{queffelec2010substitution}).\\

We expect that the methods introduced in this work pave the way for obtaining similar results in a much larger class of substitutions, for example, substitutions of constant length. We emphasize the fact that there is no systematic method to calculate the Lyapunov exponents in the general case. Some particular cases where this is possible are the case of bijective abelian substitutions (see \cite{baake2019renormalisation}) and constant length substitutions on two letters (see \cite{manibo2017lyapunov,BaakeBinary}).\\

For the time being, we have restricted our attention to bijective subtitutions on two letters, the Thue-Morse substitution being the most prominent example. In this case, we obtain some rigidity on the behavior of the top Lyapunov exponent with respect to the subshift factors, stated in Theorem \ref{main}. A special property of bijective substitutions on two letters we take advantage of is that we may represent its twisted Birkhoff sums as a Riesz product. Namely, for the Thue-Morse substitution, the twisted Birkhof sum over a fixed point may be expressed as
\[
p_n(\omega) = \prod_{j=0}^{n-1} 2 \sin(\pi\{\omega2^j\}).
\]

A lot of attention has been given to the understanding of the asymptotics of $L^p$ norms of $p_n$: part of the proof of a prime number theorem for the sum-of-digits function, due to Maduit and Rivat in \cite{mauduit2010probleme}, uses the asymptotic
\[
\norm{p_n}_1 \sim 2^{n\delta}, \text{ with } \delta = 0.40325...
\]
A survey with the link between these products and the Thue-Morse sequence is found in \cite{queffelec2018questions}. We will be interested in the asymptotics of $p_n$ for fixed $\omega$ as $n$ goes to infinity. Apparently contradictory with the latter result, the growth is subexponential for generic $\omega$. The next result answers a question implicitly asked in \cite{BaakeBinary}.
\begin{theorem}\label{subexp}
	There exists a positive constant $B$ such that for almost all $\omega$, there is a positive integer $n_0(\omega)$ such that for all $n\geq n_0(\omega)$,
	\[
	\max (p_n(\omega),p_n^{-1}(\omega)) \leq e^{B\sqrt{n\log\log(n)}}.
	\]
\end{theorem}
It might be worth to compare this result with the growth of a cocycle in the context of discrete Schrödinger operators arising from the Thue-Morse substitution found in \cite{liu2017unbounded}, where the authors show a similar subexponential behavior.\\

Theorem \ref{subexp} will imply, in particular, that the top Lyapunov exponent of the spectral cocycle for the Thue-Morse substitution vanishes, which is a known fact (see \cite{BaakeBinary}). In fact, this will be the case for all topological factors of the Thue-Morse subshift coming from an aperiodic substitution. Moreover, we will prove a more general result for all primitive aperiodic bijective substitutions on a two letter alphabet, which is our main result. Let $\rho(A)$ be the spectral radius of a matrix $A$. 
\begin{theorem}\label{main}
	Let $\zeta$ be an aperiodic, primitive and bijective subtitution on two letters of constant length $q$. If $\sigma$ is an aperiodic substitution such that the associated dynamical system is a topological factor of the one associated to $\zeta$, then the corresponding top Lyapunov exponents of the spectral cocycle satisfy
	\[
	\dfrac{\chi_{\sigma}^+(\omega)}{\log(\rho(M_\sigma))} \leq \dfrac{\chi_{\zeta}^+}{\log(q)},
	\]
	for Lebesgue almost all $\omega$, where $M_\sigma$ is the substitution matrix of $\sigma$.
\end{theorem}
As we shall see, the nonnegativity of the top Lyapunov exponent of the spectral cocycle is a general fact (Theorem \ref{Positivity}), so for the Thue-Morse substitution we have the next result.
\begin{cor}\label{main2}
	For every topological factor of the Thue-Morse subshift coming from an aperiodic substitution $\sigma$ we have $\chi_{\sigma}^+(\omega) = 0$, for Lebesgue almost all $\omega$.
\end{cor}
In fact, there is only one non-trivial subshift factor of the Thue-Morse subshift up to conjugacy (this may be deduced from \cite{coven2017topological} and \cite{coven2016computing}) which is given by the period doubling substitution. But there is no unicity of the substitution used to represent a subshift (in fact, there are infinitely many, see Theorem \ref{Dekk} below), and it is not necessarily true that two substitutions with conjugate systems have the same top Lyapunov exponent of the spectral cocycle.\\

We organize the paper as follows: in Section 2 we recall several notions around substitutions and dynamics, in particular, those related to the Thue-Morse sequence. In addition we recall the definition of the spectral cocycle (and results on its top Lyapunov exponent), the twisted Birkhoff sum and consider one generalization: the twisted correlation. To finish, we state a result about asymptotic laws of expanding maps of the unit interval which will be essential for our estimates.\\

Section 3 begins with three general results about the top Lyapunov exponent of the spectral cocycle. It continues with finer estimates for the twisted Birkhoff sums in the Thue-Morse case (in particular, showing its subexponential behavior) and for constant length substitutions on two letters. We finish this section with upper bounds for the twisted correlations of Thue-Morse. For simplicity and as a motivation, the bounds are proved in the simplest case, but presented in general.\\

In Section 4 we state the analogous results of last section on the Thue-Morse substitution to an arbitrary bijective substitution on two letters. In particular, we state Theorem \ref{TwistedCorrUnifGen} which contains upper bounds for the twisted correlations. In Section 5 we present the proof of Theorem \ref{main} based on the results of Section 4. Finally in the Appendix we present in full generality the proof of the upper bounds for the twisted correlations of bijective substitution on two letters.

\section*{Acknowledgements}
The author would like to thank J\"{o}el Rivat for the discussions around exponential sums of digital functions. The author would also like to thank Beno\^{i}t Saussol to direct us to the article \cite{dedecker2012almost} for the bounded law of iterated logarithm, and S\'ebastien Gou\"{e}zel for explaning to us the deduction of this law from an invariance principle and for reading the first manuscript. I would like to thank the two anonymous referees for their exhaustive lists of corrections, which have significantly improved the presentation of this article. Finally, I would like to thank my supervisors Alexander I. Bufetov and Pascal Hubert, for all the discussions around this problem and constant encouragement.

\section{Background}
\subsection{Substitutions}\label{background} For the basic notions on substitutions see for example \cite{queffelec2010substitution}. Let us start by fixing a positive integer $m\geq 2$ and a finite alphabet $\A = \left\{1,\dots,m\right\}$. A \textit{substitution} on the alphabet $\A$ is a map $\zeta : \A \longrightarrow \A^+$, where we use $\A^+$ to denote the set of finite (non-empty) words on $\A$. By concatenation, it is natural to extend a substitution to $\A^+$ to $\A^\N$ (one-sided sequences) or $\A^\Z$ (two-sided sequences). In particular, the iterates $\zeta^n(a) = \zeta(\zeta^{n-1}(a))$ for $a \in \A$, are well defined.\\

For a word $w\in\A^+$ let $\abs{w}$ be its length and let $\abs{w}_a$ be the number of symbols $a$ found in $w$. The \textit{substitution matrix} associated to a substitution $\zeta$ is the $m\times m$ matrix with integer entries defined by $M_{\zeta}(a,b) = \abs{\zeta(b)}_a$. A substitution is called \textit{primitive} if its substitution matrix is primitive (see \cite{queffelec2010substitution}, Chapter 5). If there exists $q\in\N, q\geq 2$ such that $\abs{\zeta(a)} = q$ for all $a\in\A$, we say that $\zeta$ is a \textit{constant-length} substitution.\\

The \textit{substitution subshift} associated to $\zeta$ is the set $X_\zeta$ of sequences $(x_n)_{n\in\Z} \in \A^\Z$ such that for every $i\in\Z$ and $k\in\N$ exist $a \in \A$ and $n\in\N$ such that $x_i\dots x_{i+k}$ is a subword of some $\zeta^n(a)$. Let $w=w_0\dots w_k$ be a word in $\A^+$ and $l\in\Z$. The \textit{cylinder} at $l$ associated to $w$ is the set
\[
[w]_l = \left\{ (x_n)_{n\in\Z} \in X_{\zeta} \,\middle| \, x_l = w_0,\dots, x_{l+k}=w_k \right\}.
\]
When $l=0$ we will omit the subindex $l$. Abusing notation, let $[w]$ be the elements of $\A^+$ that start with $w$, that is, having $w$ as a prefix. In particular, we will consider in the next sections the indicator function of a cylinder $[w]\subset X_\zeta$, which we will denote by $\mathds{1}_{[w]}$.\\

A classical result is that the $\Z$-action by the \textit{left-shift} $T((x_n)_{n\in\Z}) = (x_{n+1})_{n\in\Z}$ on $X_{\zeta}$ is minimal and uniquely ergodic when $\zeta$ is primitive (see \cite{queffelec2010substitution}, Chapter 5). From now on we only consider primitive and \textit{aperiodic} substitutions, i.e., such that there are no shift-periodic sequences in the subshift.\\

Finally, let us recall the definition of conjugacy and semiconjugacy. We say two subshifts $(X,T),(Y,T)$ are \textit{topologically conjugate} if there is a homeomorphism $\pi: X \longrightarrow Y$ such that $\pi \circ T = T\circ \pi$. If $\pi$ is only onto, then we say $X$ and $Y$ are \textit{topologically semiconjugate}, or that $Y$ is a \textit{factor} of $X$.

\subsection{Thue-Morse sequence and bijective subtitutions} Let 
\[
\zeta_{TM}: 0 \mapsto 01, 1\mapsto 10
\]
be the Thue-Morse substitution. This substitution has two fixed points, namely, $\mathbf{u}=\zeta_{TM}^{\infty}(0)$ and $\mathbf{v}=\zeta_{TM}^{\infty}(1)$. The former one is called the Thue-Morse sequence. Note that $\mathbf{v}$ is equal to $\mathbf{u}$ after changing the symbols $0$'s for $1$'s and $1$'s for $0$'s. Another characterization of this sequence which we will use along the whole text is the next arithmetic property valid for all integers $k\geq 0$:
\[
\mathbf{u}_k = s_2(k) \text{ (mod 2)},
\]
where $s_2(k)$ is equal to the sum of digits in the binary expansion of $k$, i.e., if $k = k_0 + \dots + k_{n-1}2^{n-1}$ with $k_j \in \{0,1\}$, then $s_2(k) = k_0 + \dots + k_{n-1}$.\\

The Thue-Morse substitution is a particular case of \textit{bijective} substitutution.
\begin{definition}
	A substitution $\zeta$ of constant length $q$ is bijective if for all $j \in \{1,\dots, q\}$ there exists a permutation $\pi_j$ of $\A$ such that for all $a\in\A$  we have $\zeta(a)_j = \pi_j(a)$. In particular, for a fixed $j \in \{1,\dots, q\}$ all letters $\zeta(a)_j$, for $a\in\A$, are different.
\end{definition}
As we will see in Section 4, an analogue of the arithmetic property of Thue-Morse is shared by all bijective substitution on two letters (see Corollary \ref{arithmeticBij}).
\begin{example}
	Let us consider the alphabets $\A = \{1,2,3,4,5\}$ and $\mathcal{B}=\{0,1\}$ and the following substitutions:
	\begin{align*}
	\zeta_1: \begin{cases} 
	1\mapsto 123325 \\
	2\mapsto 235414 \\
	3\mapsto 544231 \\
	4\mapsto 311142 \\
	5\mapsto 452553
	\end{cases} \,
	\zeta_2: \begin{cases} 
	1\mapsto 123323 \\
	2\mapsto 235414 \\
	3\mapsto 544231 \\
	4\mapsto 311142 \\
	5\mapsto 452553
	\end{cases} \,
	\zeta_3: \begin{cases} 
	0\mapsto 01101011 \\
	1\mapsto 10010100
	\end{cases} \,
	\zeta_3: \begin{cases} 
	0\mapsto 01010 \\
	1\mapsto 11100
	\end{cases}
	\end{align*}
	
The substitutions $\zeta_1,\zeta_2$ are of constant length ($q=6$) on $\A$. The substitution $\zeta_1$ is bijective and $\zeta_2$ is not bijective, since for $j=6$ we have $\zeta_2(1)_6 = \zeta_2(5)_6 = 3$. The substitutions $\zeta_3,\zeta_4$ are of constant length ($q=8$ and $q=5$ respectively) on $\mathcal{B}$. The substitution $\zeta_3$ is bijective and $\zeta_4$ is not bijective, since for $j=2$ we have $\zeta_4(0)_2 = \zeta_4(1)_2 = 1$.
\end{example}
\subsection{Conjugacy and factor list of a substitution subshift}\label{conjugacy}
Here we recall some results on factors of substitution subshifts. First, we recall the classical result of Curtis-Hedlund-Lyndon on topological factor maps on general subshifts. For a shift space $X$, let $\mathcal{L}_r(X)$ be the set of all possible words appearing in elements of $X$ of length $r\geq 0$.

\begin{definition}
A \textit{sliding block code} $\pi$ is a map between shift spaces $X\subseteq\A^\Z$, $Y\subseteq \mathcal{B}^\Z$ such that there exists an integer $r\geq0$ and a map $\widehat{\pi}:\mathcal{L}_{2r+1}(X)\longrightarrow \mathcal{B}$ such that the image $(y_n)_{n\in\Z}\in Y$ by $\pi$ of $(x_n)_{n\in\Z}\in X$ is given by the rule
\[
y_n = \widehat{\pi}(x_{n-r}\dots x_n \dots x_{n+r}) \quad \text{for all } n\in\Z.
\]
\end{definition}
\begin{theorem}[see \cite{lind_marcus_1995}, Theorem 6.2.9] \label{CHL}
A function $\pi: X \longrightarrow Y$ is a sliding block code if and only if it is continuous and commutes with the shift map $T$.
\end{theorem}
For substitution subshifts we have the next two general results.
\begin{theorem}[\cite{durand2000linearly}]\label{du00}
	Subshift topological factors of substitution systems are topologically conjugate to substitution subshifts.
\end{theorem}
\begin{theorem}[\cite{durand1999substitutional}]
	There exist finitely many subshift topological factors of a substitution subshift up to topological conjugacy.
\end{theorem}
Deciding when two substitution systems are topologically conjugate or semiconjugate (one is a factor of the other) is a problem recently solved (the conjugacy problem follows from the second one by coalesence of substitution subshifts).
\begin{theorem}[\cite{durand2022decidability}]\label{DuLe} 
	For two uniformly recurrent substitution subshifts $(X,T)$ and $(Y,T)$, it is decidable wheter they are semiconjugate. Moreover, if $(Y,T)$ is aperiodic, then there exists a computable constant $r$ such that for any factor $\pi: X \longrightarrow Y$ there exist $k\in \Z$ and a factor $\pi':X \longrightarrow Y$ of radius less than $r$, such that $\pi = T^k\circ \pi'$.
\end{theorem}
It is tempting to think we may list all factors from a given substitution system, but the radius $r$ in the last result depends strongly on the factor system. On the other hand, in the constant length case, this is partially solved: it is possible to give the list of all factors of constant length (see \cite{coven2017topological,durand2022decidability}). Examples of conjugate substitution systems to the Thue-Morse subshift which are not constant length are given in \cite{dekking2014structure}.
\begin{theorem}[\cite{dekking2014structure}]\label{Dekk}
	There exist infinitely many non-constant length, primitive, injective substitutions with Perron-Frobenius eigenvalue equal to 2, conjugate to the Thue-Morse substitution.
\end{theorem}
\begin{example}
	Consider the substitutions $\zeta_1$ and $\zeta_2$ defined below. Both are conjugate to the Thue-Morse substitution. The substitution $\zeta_3$ (period doubling) defines a non-trivial topological factor.
	\begin{align*}
	\zeta_1:& \: 0 \mapsto 01, 1\mapsto 20, 2\mapsto 10,\\
	\zeta_2:& \: 0 \mapsto 012, 1\mapsto 02, 2\mapsto 1,\\
	\zeta_3:& \: 0 \mapsto 01,\: 1\mapsto 00.
	\end{align*}
\end{example}

\subsection{Solutions of $s_2(k+a)-s_2(k)=d$}\label{Secsolutions} Here we briefly recall some results from \cite{emme2017asymptotic}. Consider the equation $s_2(k+a)-s_2(k)=d$, for $a \in \N$ and $d\in\Z$. Let $\mathcal{S}_{a,d}$ be its solution set. Note that if $d>\ceil{\log_2(a)}$, then $\mathcal{S}_{a,d}=\emptyset$. Also, $d<-n$ implies $\mathcal{S}_{a,d}\cap\{0,\dots,2^n-1\}=\emptyset$. These will be trivial cases of the next lemma.\\

To change between words and numbers represented by digits consider the next notation: $\underline{k}_2$ denotes the word associated to the digits of $k$ in base 2, i.e., if $k = k_0 + \dots + k_{n-1}2^{n-1}$ then, $\underline{k}_2 = k_0\dots k_{n-1}\in\{0,1\}^*$. Similarly, for a word $w=w_0\dots w_{n-1} \in  \{0,1\}^*$, let $\overline{w}^2$ be the number $w_0 + \dots + w_{n-1}2^{n-1}$.
\begin{lemma}[see \cite{emme2017asymptotic}, Lemma 2.1.1]{\label{EPwords}}
	There exists a finite set of words $\mathcal{P}_{a,d} = \{\mathfrak{p}^d_a(1),\dots,\mathfrak{p}^d_a(s)\} \subset \{0,1\}^*$, such that
	\[
	k\in \mathcal{S}_{a,d} \iff \underline{k}_2 \in \bigcup_{i=1}^s \: [\mathfrak{p}^d_a(i)],
	\]
\end{lemma}
Let $\norm{\mathcal{P}_{a,d}}$ be the length of the longest word of $\mathcal{P}_{a,d}$. We will keep using $\mathcal{P}_{a,d}$ to denote the set made from elements from the old $\mathcal{P}_{a,d}$ after adding all possible suffixes in $\{0,1\}$ to all words which are not of maximal length $\norm{\mathcal{P}_{a,d}}$, making all the words of the new set of length $\norm{\mathcal{P}_{a,d}}$.
\begin{lemma}\label{TreeBound}
	For all $a \in \N$ and $d\in\Z$, the set $\mathcal{P}_{a,d}$ (if it is non-empty) satisfies:
	\begin{itemize}
		\item $\#\mathcal{P}_{a,d} \leq 8a^3$.
		\item $\norm{\mathcal{P}_{a,d}} = 2\ceil{\log_2(a)} + \abs{d}$.
	\end{itemize}
\end{lemma}
\begin{proof}
	Both claims come direct from the proof of Lemma \ref{EPwords}: we generalize the proof to strongly $q$-additive functions in Lemma \ref{lemmaWords} (see also the proof of Corollary \ref{TreeBoundGeneral}). Here we sketch the proof to use the explicit constants later in the particular case of the Thue-Morse substitution. For the second point, since $a$ has at most $\ceil{\log_2(a)}$ digits in base 2, the longest word of $\mathcal{P}_{a,d}$ has a prefix of length at most $\ceil{\log_2(a)}$ followed by a suffix from some $\mathcal{P}_{1,d'}$, and $\abs{d'}\leq \ceil{\log_2(a)} + \abs{d}$. In summary, $\norm{\mathcal{P}_{a,d}} \leq 2\ceil{\log_2(a)} + \abs{d}$.\\
	
	By our assumption, all words from $\mathcal{P}_{a,d}$ have the same length, i.e., we can take $\norm{\mathcal{P}_{a,d}} = 2\ceil{\log_2(a)} + \abs{d}$. Finally, each suffix added is of length at most $2\ceil{\log_2(a)}$ and $a$ has at most $\ceil{\log_2(a)}$ digits in base 2, which defines $2^{\ceil{\log_2(a)}}$ possible prefixes for the words in $\mathcal{P}_{a,d}$. In summary, we have $\#\mathcal{P}_{a,d} \leq 2^{\ceil{\log_2(a)}}2^{2\ceil{\log_2(a)}} \leq 8a^3$.
\end{proof}

\subsection{Mixing coefficients and bounded law of iterated logarithm}A major technical tool needed to obtain the bounds of subsections 3.2 and 3.3 is a bounded iterated logarithm law for uniformly expanding systems, in the case of unbounded observables. Although the proof of this result does not play a role for our work, we comment the context for completeness. \\

Following \cite{dedecker2012almost}, the result comes from the study of some mixing coefficients, definition we recall next. For an integrable random variable $X$ of a probability space $(\Omega,\mathcal{C},\mathbb{P})$, set $X^{(0)} = X - \E(X)$. For a real random vector $(X_1,\dots,X_k)$ and a $\sigma$-algebra $\mathcal{F}\subset\mathcal{C}$, define
\[
\phi(\mathcal{F},X_1,\dots,X_k) = \sup_{(x_1,\dots,x_k)\in\R^k} \left|\left| \E\left( \left. \prod_{j=1}^k \left( \mathds{1}_{\{X_j\leq x_j\}} \right)^{(0)} \, \right| \mathcal{F} \right)^{(0)} \right|\right|_{\infty}.
\]
Let $R$ be a bijective bimeasurable map of $\Omega$. For a process $\mathbf{X} = (X_j)_{j\geq0}$ such that $X_i = R^i \circ X_0$ and $X_0$ is measurable with respect to a $\sigma$-algebra $\mathcal{F}$ that satisfies $\mathcal{F}\subset R^{-1}\mathcal{F}$, define
\[
\phi_{k,\mathbf{X}}(n) = \max_{1\leq l\leq k} \: \sup_{n\leq i_1<\dots<i_l} \phi( \mathcal{F},X_{i_1},\dots,X_{i_l} ).
\]

The main importance of these coefficients is that a suitable control on them implies a bounded law of iterated logarithm for the Birkhoff sums of uniformly expanding maps of the interval (being one of its inverse branches the transformation $R$) and observables $f$ in the closure of the family of functions $\mathrm{Mon}_2(M,\nu)$ (below we denote this closure by $\mathrm{Mon}^c_2(M,\nu)$), where $\text{Mon}_2(M,\nu)$ ($M>0$, $\nu$ a measure on the unit interval) is the set of functions $g:[0,1]\longrightarrow\R$ such that $\nu(\abs{g}^2)\leq M^2$ (square-integrable with respect to the measure $\nu$) and $g$ is a finite sum of functions which are monotonic in some interval and zero elsewhere.
\begin{theorem}[\cite{dedecker2012almost}]\label{thmGouezel}
	Let $S$ be a uniformly expanding map of the unit interval with an absolutely continuous invariant measure $\nu$. Then, for any $M>0$ and $f \in \mathrm{Mon}^c_2(M,\nu)$, there exists a nonnegative constant $A$ such that 
	\[
	\sum_{n=1} \dfrac{1}{n} \nu\left(\max_{1\leq k\leq n} \left|\sum_{i=0}^{k-1} (f\circ S^i - \nu(f)) \right|\geq A\sqrt{n\log\log(n)}  \right) < \infty.
	\]
	This convergence implies a bounded law of iterated logarithm: for $\nu$-almost every $\omega$,
	\[
	\limsup_{n\to\infty} \dfrac{1}{\sqrt{n\log\log(n)}}\left| \sum_{i=0}^{n-1} (f\circ S^i(\omega) - \nu(f)) \right| \leq A.
	\]
\end{theorem}

\subsection{Spectral cocycle and top Lyapunov exponents} As usual, we use the notation $e(x) := e^{2\pi i x}$. Following \cite{bufetov2020spectral}, let $\zeta$ be a substitution over $\A = \{1,\dots,m\}$ and let $\zeta(a) = w_1\dots w_{k_a}$. 
\begin{definition}\label{generalCocycle}
	The \textit{spectral cocycle} is the cocycle over $\mathbb{T}^m = (\R/\Z)^m$ with $(a,b)\in\A^2$ entry given by
	\begin{align*}
	\mathscr{C}_{\zeta}(\xi,1)(a,b) &= \sum_{j=1}^{k_a} \delta_{w_j,b} \:e(\xi_{w_1} + \dots + \xi_{w_{j-1}}), \quad \xi \in \mathbb{T}^m \\
	\mathscr{C}_{\zeta}(\xi,n) &=  \mathscr{C}_{\zeta}((M^T_{\zeta})^{n-1}\xi,1) \dots \mathscr{C}_{\zeta}(\xi,1).
	\end{align*}
\end{definition}
The original definition actually uses a minus sign in the exponentials, but it is just a convention and it will not affect the calculations. Since we will only consider the spectrum of the substitution with the $\Z$-action, we will restrict this cocycle to the diagonal, i.e, for $\xi = \omega \vec{1}$, with $\omega \in [0,1)$. In the case of constant length substitutions (of length $q$), the base dynamics is the $q$-times map $\omega \mapsto q\omega$ (mod 1) over $\mathbb{T}$.
\begin{example}
	Let $\zeta_{TM}: 0 \mapsto 01, 1\mapsto 10$ be the Thue-Morse substitution. Then, for $(\xi_0,\xi_1)^T = \omega \vec{1}$, 
	\begin{align*}
	\mathscr{C}_{\zeta_{TM}}(\omega\vec{1},1)&=\begin{pmatrix}1&e(\omega)\\e(\omega)&1\end{pmatrix}, \\
	\mathscr{C}_{\zeta_{TM}}(\omega\vec{1},n) &= \begin{pmatrix}1&e(2^{n-1}\omega)\\e(2^{n-1}\omega)&1\end{pmatrix} \dots \begin{pmatrix}1&e(\omega)\\e(\omega)&1\end{pmatrix}.
	\end{align*}
\end{example}
The pointwise $\textit{top Lyapunov exponent}$ of the spectral cocycle $\mathscr{C}_{\zeta}$, is defined for $\omega\in [0,1)$ as
\[
\chi_\zeta^+(\omega) = \limsup_{n\to\infty} \dfrac{1}{n} \log \norm{\mathscr{C}_{\zeta}(\omega\vec{1},n)}.
\]
By the Furstenberg-Kesten theorem, this is in fact a limit for almost every $\omega$. In the case of a constant length substitution, since the $q$-times map $x \to qx$ (mod 1) ($q\in\N, q\geq 2$) is an ergodic transformation for the Lebesgue measure (denoted by Leb), the Furstenberg-Kesten theorem implies that $\chi_{\zeta}^+(\omega)$ is almost everywhere constant, value we will denote by $\chi_\zeta^+$. A reference for general theory on cocycles is \cite{viana2014lectures}.

\subsection{Results on binary constant length substitutions} We recall some results proved in \cite{BaakeBinary} concerning constant length substitutions on two letters. The main result gives an expression of the top Lyapunov exponent as the logarithmic Mahler measure of certain polynomial.
\begin{theorem}[\cite{BaakeBinary}, Theorem 1.1]\label{mainBaake}
	Let $\zeta$ be a substitution of length $q$ on two letters. Then 
	\[
	\chi^+_\zeta = \mathfrak{m}(P).
	\]
	for some polynomial $P\in\Z[X]$ with integer coefficients and degree $q-1$, where $\mathfrak{m}(P)$ denotes the logarithmic Mahler measure of $P$.
\end{theorem}
In fact, in \cite{BaakeBinary} they also prove that the other Lyapunov exponent vanishes, but we will not use this fact. A crucial observation to prove this result is the next
\begin{lemma}[\cite{BaakeBinary}, Remark 4.3]\label{product}
	Let $\zeta$ be a substitution of length $q$ on two letters, $v = (1,-1)^T$ and $P$ the polynomial of Theorem \ref{mainBaake}. Then
	\[
	\mathscr{C}_{\zeta}(\omega\vec{1},n) v = \prod_{k=0}^{n-1} P(e(q^k\omega)) v
	\]
	and 
	\[
	\chi^+_\zeta = \lim_{n\to\infty} \dfrac{1}{n} \log \norm{\mathscr{C}_{\zeta}(\omega\vec{1},n) v}.
	\]
\end{lemma}
Let us mention that the polynomial $P$ is explicit and depends only on the bijective columns of the substitution, that is, on elements $j\in\{1,\dots,q\}$ such that $\zeta(0)_j \neq \zeta(1)_j$.

\subsection{Twisted Birkhoff sums and twisted correlations} In this subsection we define a special kind sum appearing in the calculations of the top Lyapunov exponents for factors of Thue-Morse, but whose study might be interesting by itself. First, we recall the notion of twisted Birkhoff sum:
\begin{definition} Let $(X,\mu,T)$ be a measure-preserving system, $f$ a measurable function and $\omega\in[0,1)$. The \textit{twisted Birkhoff sum} $S^{f}_N(\omega,x)$ (at time $N$, of parameter $\omega$) of $f$ at a point $x \in X $ is defined by
	\[
	S^{f}_N(\omega,x) = \sum_{k=0}^{N-1} f(T^kx)e(k\omega).
	\]
\end{definition}
This type of sum (or integral in the continuous case) is of major importance in several recent works like \cite{marshall2020modulus,trevino2020quantitative,JEP_2021__8__279_0,avila2023quantitative} since uniform (on $x\in X$) upper bounds allow one to obtain lower bounds for the dimension of spectral measures of the dynamical systems involved: substitutions, interval exchange transformations, suspensions over S-adic systems (including translation surfaces) and higher dimensional tilings.\\

A generalization of this kind of sum is the next concept. For more generalizations of multiple ergodic sums see for example \cite{eisner2013linear} and references therein.
\begin{definition}\label{defTwsited}
	Let $(X,\mu,T)$ be a dynamical system, $f$ a measurable function. Let $a_1<\dots< a_t$ positive integers and $\omega\in[0,1)$. The (accumulated) \textit{twisted correlation} $C^{f}_N(a_1,\dots,a_t,\omega,x)$ (at time $N$, of parameters $a_1,\dots,a_t$ and $\omega$) of $f$ at a point $x \in X $ is defined by
	\[
	C^{f}_N(a_1,\dots,a_t,\omega,x) = \sum_{k=0}^{N-1} f(T^kx)f(T^{k+a_1}x)\dots f(T^{k+a_t}x)e(k\omega).
	\]
\end{definition}
We will study this concept in the case of subshifts arising from bijective substitutions on two letters, and for the indicator functions $\mathds{1}_{[0]},\mathds{1}_{[1]}$. It turns out that better algebraic properties emerge if we consider the function $\mathds{1}_{[0]}-\mathds{1}_{[1]}$, and the twisted correlation of the former functions may be bounded from above using a finite combination of twisted correlations of the latter function. For this last function we denote its twisted correlations by $C^{\pm}_N(a_1,\dots,a_t,\omega,\mathbf{x})$. For example, if we consider $\mathbf{u}$ the Thue-Morse sequence (note that we are only interested in the nonnegative coordinates), the twisted correlation will be of the form
\[
C^{\pm}_N(a_1,\dots,a_t,\omega,\mathbf{u}) = \sum_{k=0}^{N-1} (-1)^{s_2(k)+s_2(k+a_1) + \ldots +s_2(k+a_t)}e(k\omega).
\]
Of course, for the other fixed point of the Thue-Morse substitution denoted by $\mathbf{v}$, we have the next simple relation for the twisted correlation (recall that $\mathbf{v}$ is just $\mathbf{u}$ after swaping 0's and 1's):
\[
C^{\pm}_N(a_1,\dots,a_t,\omega,\mathbf{u}) = (-1)^{t+1} C^{\pm}_N(a_1,\dots,a_t,\omega,\mathbf{v}).
\]

Where do these sums come up? They appear naturally when studying the twisted Birkhoff sums of (topological) subshift factors of a substitution (on two letters) subshift: by Theorem \ref{CHL}, every factor $\pi$ comes from a map $\widehat{\pi}: \mathcal{L}_{2r+1} \to \mathcal{B}$, for some $r\geq0$. By Theorem \ref{du00} factors of substitution systems are conjugate to substitution subshifts. Let $a,b \in \mathcal{B}$, $\sigma$ the factor substitution and $\mathbf{w}$ any element of $X_\sigma$ starting with $\sigma^n(a)$. If $\mathbf{x}\in \pi^{-1}(\{\mathbf{w}\})$, we have the next relation for $f=\mathds{1}_{[b]}$:
\begin{align*}
S^f_{\abs{\sigma^n(a)}}(\omega,\mathbf{w}) &= \sum_{w: \widehat{\pi}(w) = a} \sum_{k=0}^{\abs{\sigma^n(a)}-1} \mathds{1}_{[w_{-r}\dots w_{r}]_{-r}}(T^k \mathbf{x}) e(k\omega) \\
&= \sum_{w: \widehat{\pi}(w) = a} \sum_{k=0}^{\abs{\sigma^n(a)}-1} \mathds{1}_{[w_{-r}]}(T^{k-r}\mathbf{x})\cdot\ldots\cdot\mathds{1}_{[w_{r}]}(T^{k+r}\mathbf{x}) e(k\omega)
\end{align*}

Note in the last equality there are some terms depending on negative coordinates, but the number of terms is bounded (in fact computable according to Theorem \ref{DuLe}) as $n$ goes to infinity, so it will not interfere with determining asymptotic bounds. Using the simple relation between the indicator functions, we obtain the twisted Birkhoff sum of the factor as a sum of different twisted correlation in the Thue-Morse subshift. The details are explained in Section 6.\\

\begin{remark}
	The case $\omega = 0$, i.e., of correlations of higher order, has been recently studied in \cite{baake2023correlations}.
\end{remark}

\section{Results}
\subsection{Top Lyapunov exponent of a general substitution} A simple observation leads to a basic fact about the top Lyapunov exponent of an arbitrary substitution. Here and in the proof of Proposition \ref{productbound}, we use the following simple fact: consider $f(x) = \log(\abs{2\cos(\pi x)})$, $g(x) = \log(\abs{2\sin(\pi x)})$. It is an excercise to check that $f,g\in L^1([0,1],\text{Leb})$ and both have mean equal to zero.
\begin{prop}\label{Positivity}
	Let $\zeta$ be an arbitrary substitution on $\A$. Then, $\chi_\zeta^+(\omega)\geq 0$ for Lebesgue almost every $\omega \in [0,1)$.
\end{prop}
\begin{proof}
	Note that for the line of the matrix $\mathscr{C}_{\zeta}(\omega,n)$ indexed by $a\in\A$, all complex exponentials $1,\ldots,e((\abs{\zeta^n(a)}-1)\omega)$ appear as a summand in some entry of this line. From this fact we deduce rapidly,
	\begin{align*}
	\norm{ \mathscr{C}_{\zeta}(\omega,n) }_{\infty} &\geq \sum_{b\in\A} \abs{\mathscr{C}_{\zeta}(\omega,n)(a,b) }\\ 
	&\geq \abs{1+\ldots+e((\abs{\zeta^n(a)}-1)\omega)} \\
	&= \abs{2\cos(\pi\{\abs{\zeta^n(a)}\omega\})}/\abs{e(\omega) - 1}.
	\end{align*} 
	A uniform distribution result may serve to conclude. Alternatively, an elementary argument is the next one: call $f(x) = \log(\abs{2\cos(\pi x)})$. We will conclude the proof by means of the next Borel-Cantelli based argument: let $\epsilon>0$ and $\mathfrak{G}_n(\epsilon) = \mathfrak{G}_n = \{\omega\in[0,1) | f(\{\abs{\zeta^n(a)}\omega\}) \geq n\epsilon\}$. Since the Lebesgue measure is invariant for all $q$-times maps,
	\begin{align*}
	\text{Leb}(\mathfrak{G}_n) &= \text{Leb}\left(\{\omega\in[0,1) \,|\, f(\{\abs{\zeta^n(a)}\omega\}) \geq n\epsilon\}\right)\\
	&= \text{Leb}\left( \{\omega\in[0,1) \,|\, \abs{f(\omega)} \geq n\epsilon\} \right) \\ 
	&= \sum_{k\geq n} \text{Leb}\left( \left\{\omega\in[0,1) \,|\, k\leq\dfrac{\abs{f(\omega)}}{\epsilon} < k+1\right\} \right).
	\end{align*}
	Then,
	\begin{align*}
	\sum_{n\geq 1}\text{Leb}(\mathfrak{G}_n) &= \sum_{n\geq1} \sum_{k\geq n} \text{Leb}\left( \left\{\omega\in[0,1) \,|\, k\leq\dfrac{\abs{f(\omega)}}{\epsilon} < k+1\right\} \right)\\
	&= \sum_{k\geq1} k\,\text{Leb}\left( \left\{\omega\in[0,1) \,|\, k\leq\dfrac{\abs{f(\omega)}}{\epsilon} < k+1\right\} \right)\\
	&\leq \int_{0}^{1} \dfrac{\abs{f(x)}}{\epsilon}dx < +\infty.
	\end{align*}
	From Borel-Cantelli lemma we deduce that for any $\epsilon>0$, $\text{Leb}\left(\liminf_n \mathfrak{G}_n(\epsilon) \right)= 1$. Considering $\mathfrak{G} = \bigcap_i \liminf_n \mathfrak{G}_n(1/i)$ we have Leb($\mathfrak{G}$) = 1 and for any $\omega\in \mathfrak{G}$, $\lim_n \dfrac{f(\{\abs{\zeta^n(a)}\omega\})}{n} = 0$. 
\end{proof}
\begin{remark}
	A similar result is proved in Theorem 3.34 of \cite{baake2019renormalisation}, but for almost every parameter in the line generated by the Perron-Frobenius eigenvector instead of the vector $\vec{1}$ as in our case. It is worth mentioning also that the top Lyapunov exponent of the general cocyle as defined in Definition \ref{generalCocycle} is nonnegative if we assume the substitution matrix has no eigenvalues that are roots of unity, as shown in Corollary 2.8 of \cite{bufetov2022substitution}.
\end{remark}

To finish this section we prove a result regarding the rest of the Lyapunov spectrum of a constant length substitution. 
\begin{prop}
	Let $\zeta$ be a substitution of constant length $q$ on an alphabet $\A$ of cardinality $d$. Then there exists a Lyapunov exponent $\chi_j(\omega) = 0$ almost surely.
\end{prop}

\begin{proof}
	Indeed, note that the vector $\vec{1}\in \R^d$ is a rigth eigenvector of $\mathscr{C}_{\zeta}(\omega\vec{1},n)$ with eigenvalue
	\[
	\prod_{j=0}^{n-1} 1 + e(q^j\omega) + \dots + e(q^{(d-1)j}\omega)= \prod_{j=0}^{n-1} P_d(e(q^j\omega)).
	\]
	We know $\mathfrak{m}(P_d) = 0$, since it is the product of cyclotomic polynomials. Therefore, if $\vec{1}\in W_{j}$ for some $j\in{1,\dots,d}$ in the filtration, by the Birkhoff ergodic theorem
	\[
	\lim_{n\to\infty} \dfrac{1}{n}\log \norm{\mathscr{C}_{\zeta}(\omega\vec{1},n)\vec{1} } = \lim_{n\to\infty} \dfrac{1}{n} \sum_{j=0}^{n-1} \log(\abs{P_d(q^j\omega)}) = \mathfrak{m}(P_d) = 0,
	\]
	for Lebesgue almost every $\omega$. That is, the Lyapunov exponent $\chi_j(\omega) = 0$ almost surely.
\end{proof}
\subsection{The Thue-Morse top Lyapunov exponent}
The fact that the top Lyapunov exponent for Thue-Morse is equal to zero is not a new result, in fact the result is really explicit for all constant length substitutions on two letters: see \cite{BaakeBinary}. As hinted in the introduction, this is related to the growth of the Thuse-Morse polynomials. In this subsection we will recall the proof of this fact and then prove an explicit subexponential upper bound for the Thue-Morse polynomials.\\

Let $v = (1,-1)^T$ and note that $v$ is an eigenvector for all matrices appearing in the cocycle:
\begin{align*}
	\mathscr{C}_{\zeta_{TM}}(\omega\vec{1},n) v &= \begin{pmatrix}1&e(2^{n-1}\omega)\\e(2^{n-1}\omega)&1\end{pmatrix} \dots \begin{pmatrix}1&e(\omega)\\e(\omega)&1\end{pmatrix} v\\
	&= \prod_{j=0}^{n-1} \left(1-e(\omega2^j)\right) v.
\end{align*}
As pointed out in \cite{BaakeBinary}, it is enough to study this direction to calculate the top Lyapunov exponent. Set $f(\omega) = 2\sin(\pi \omega)$ ($= \abs{1-e(\omega)}$). Since $S:[0,1) \longrightarrow [0,1)$ defined by $Sx = 2x$ (mod 1) is ergodic for the Lebesgue measure, we have that for almost every $\omega$,
\begin{align*}
\dfrac{1}{n}\log \norm{\mathscr{C}_{\zeta_{TM}}(\omega\vec{1},n)} &=
\dfrac{1}{n}\log \prod_{j=0}^{n-1} f(S^n\omega) + \dfrac{1}{n}\log(\norm{v})\\
&= \dfrac{1}{n} \sum_{j=0}^{n-1} \log(f(S^n\omega)) + \dfrac{1}{n}\log(\norm{v})\\
&\longrightarrow \int_0^1 \log(f(x))dx = 0.
\end{align*}
In fact, we can show a finer estimate for the growth of the product $p_n(\omega) := \prod_{j=0}^{n} 2\sin (\pi\{\omega2^j\})$.
\begin{prop}\label{productbound}
	There exists a positive constant $B$ such that for almost all $\omega$, there is a positive integer $n_0(\omega)$ such that for all $n\geq n_0(\omega)$,
	\[
	\max (p_n(\omega),p_n^{-1}(\omega)) \leq e^{B\sqrt{n\log\log(n)}},
	\]
	where $p_n(\omega)= \prod_{j=0}^{n} 2\sin(\pi\{\omega 2^j\})$.
\end{prop}
\begin{proof}
	Fix $\epsilon>0$. Consider the function of the unit interval $f(x)= \log (2\sin(\pi x))$. One can check that $\text{Leb}(f)=0$ and $\text{Leb}(\abs{f}^2)\leq 2^2$. Set $f_1,f_2$ as 
	\begin{align*}
	&f_1(x) = 2\log (2\sin(\pi x)) \text{ if } x\in(0,1/2] \text{ and } f_1(x) = 0 \text{ otherwise},\\
	&f_2(x) = 2\log (2\sin(\pi x)) \text{ if } x\in[1/2,1) \text{ and } f_2(x) = 0 \text{ otherwise}.
	\end{align*}
	Then, $f = \dfrac{1}{2}f_1 + \dfrac{1}{2}f_2$ and $f \in \text{Mon}_2(2,\text{Leb})$. We may apply Theorem \ref{thmGouezel} to the doubling map $S$ and the function $f$: we get for almost every $\omega$ there exists $n_0(\omega)\geq1$ such that for every $n\geq n_0(\omega)$,
	\[
	\left| \sum_{j=0}^{n}f\circ S^j(\omega) \right| \leq \underbrace{(A+\epsilon)}_{=: B}\sqrt{n\log\log(n)}.
	\]
	To obtain the desired inequality we take exponential both sides, by noticing that
	\[
	e^{\abs{x}} = e^{\max(x,-x)} = \max(e^x,e^{-x}).
	\]
\end{proof}
\subsection{Top Lyapunov exponent for a binary constant length substitution} We generalize the calculations made in the last subsection in this more general setting: we will show an explicit subexponential deviation from the expected exponential growth for the spectral cocycle.
\begin{theorem}\label{deviation}
	Let $\zeta$ be a substitution of length $q$ on two letters and $I = \chi^+_{\zeta}$. There exists a positive constant $B$ such that for almost all $\omega$ there exists a positive integer $n_0(\omega)$ such that for all $n\geq n_0(\omega)$,
	\[
	e^{-B\sqrt{n\log\log(n)}} \leq \dfrac{\norm{\mathscr{C}_{\zeta}(\omega\vec{1},n)}}{e^{nI}} \leq  e^{B\sqrt{n\log\log(n)}}.
	\]
\end{theorem}
\begin{proof} 
	Let $P$ the integer polynomial appearing in Lemma \ref{product} (which can be deduced directly from $\zeta$ as explained in \cite{BaakeBinary}) and $v = (1,-1)^T$. The real function $f:\omega \mapsto \log(\abs{P(e(\omega))})$ is defined everywhere in $[0,1]$ except maybe at $q$ points, since the trigonometric polynomial $P(e(\omega))$ has degree $q-1$.\\
	
	The reader may check that $f\in L^1([0,1], \text{Leb})$, $\text{Leb}(f) = \mathfrak{m}(P) = I$ (by Theorem \ref{mainBaake}) and $f\in L^2([0,1],\text{Leb})$, i.e., there exists $M>0$ such that $\text{Leb}(f^2)<M^2$. We may decompose $f$ into a finite sum of increasing and decreasing functions (as in the case of the Thue-Morse substitution in last section) taking into account the zeros and critical points of $P(e(\omega))$. In summary, $f\in\text{Mon}_2(M,\text{Leb})$ and by Theorem \ref{thmGouezel}, there exists $B>0$ such that for Lesbesgue almost every $\omega\in[0,1)$ there exists $n_0(\omega)\geq1$ such that for every $n\geq n_0(\omega)$,
	\[ 
	\left| \log\norm{\mathscr{C}_{\zeta}(\omega\vec{1},n)} - n\mathfrak{m}(P) \right| \leq B\sqrt{n\log\log(n)},
	\]
	i.e.,
	\[
	e^{-B\sqrt{n\log\log(n)}} \leq \dfrac{\norm{\mathscr{C}_{\zeta}(\omega\vec{1},n)}}{e^{nI}} \leq  e^{B\sqrt{n\log\log(n)}}.
	\]
	where $I = \chi^+_{\zeta} =\mathfrak{m}(P)$.
\end{proof}


\subsection{Upper bounds for twisted correlations: the Thue-Morse case}\label{twistsec}
In this subsection we look for upper bounds for the twisted correlations $C^{\pm}_{2^n}(a_1,\dots,a_t,\omega)$ for the Thue-Morse subshift. A remark on notation: given $F:\N \longrightarrow \C$, $G:\N \longrightarrow \R_+$ and some parameters $a,b>0$, we write
\[
F = O_{a,b}(G)
\]
if there exists $C = C(a,b) > 0$ depending only on $a,b$ such that for all $n\geq0$ we have
\[
\abs{F(n)} \leq CG(n).
\]
The main result of this section is the next one.
\begin{theorem}\label{TwistCorr}
	Let $C^{\pm}_{2^n}(a_1,\dots,a_t,\omega,\mathbf{u})$ be the twisted correlation defined in Definition \ref{defTwsited}, for $F = \mathds{1}_{[0]}-\mathds{1}_{[1]}$ and $\mathbf{u}$ the Thue-Morse sequence. Then
	\begin{itemize}
		\item if $t$ is even, there exists $B>0$ independent of $\omega$ and $t$ such that for almost every $\omega$,
		\[
		C^{\pm}_{2^n}(a_1,\dots,a_t,\omega,\mathbf{u}) = O_{a_t,\omega}\left(n^te^{ B\sqrt{n\log\log(n)}}\right).
		\]
		\item if $t$ is odd, for every $\epsilon>0$ and almost every $\omega$,
		\[
		C^{\pm}_{2^n}(a_1,\dots,a_t,\omega,\mathbf{u}) = O_{a_t,\omega}\left( n^{t+1+\epsilon}\right).
		\]
	\end{itemize}
\end{theorem}
For simplicity we will only prove the case $t=1$ here. For the general case, see the Appendix.\\

Equivalently, we will prove the upper bound for $\widetilde{C}_{2^n}(a,\omega,\mathbf{u}) := \sum_{k=0}^{2^n-1} [1+(-1)^{s_2(k+a) + s_2(k)}]e(k\omega)$, since the geometric sum is bounded by a constant depending on $\omega$, so it can be included in the $O_{a_t,\omega}$ sign. Let us recall that if $d>\ceil{\log_2(a)}$, then $\mathcal{S}_{a,d}=\emptyset$; and $d<-n$ implies $\mathcal{S}_{a,d}\cap\{0,\dots,2^n-1\}=\emptyset$. Recall also that for a word $w=w_0\dots w_{n-1} \in  \{0,1\}^*$, $\overline{w}^2$ is the integer $w_0 + \dots +w_{n-1}2^{n-1}$. Finally, note that we can change $s_2(k+a) + s_2(k)$ for $s_2(k+a) - s_2(k)$ because it is in the exponent of $-1$.

\begin{align*}
\widetilde{C}_{2^n}(a,\omega,\mathbf{u}) &= \sum_{k=0}^{2^n-1} [1+(-1)^{s_2(k+a) - s_2(k)}]e(k\omega)\\
&=2\sum_{d=-\ceil{\log_2(a)/2}}^{\floor{n/2}} \sum_{\substack{ k\in\mathcal{S}_{a,-2d} \\ 0\leq k\leq 2^n-1}} e(k\omega) \\
&= 2\sum_{d=-\ceil{\log_2(a)/2}}^{\floor{n/2}} \:\: \sum_{\substack{ \underline{k}_2 \in \bigcup_i \: [\mathfrak{p}_{a,-2d}(i)] \\ 0\leq k\leq 2^n-1}} e(k\omega) \\
&= 2\sum_{d=-\ceil{\log_2(a)/2}}^{\floor{n/2}} \sum_{i=1}^{\#\mathcal{P}_{a,-2d}} \:\: \sum_{k=0}^{2^{n-\norm{\mathcal{P}_{a,-2d}}}-1} e\left(\left(2^{\norm{\mathcal{P}_{a,-2d}}}k + \overline{\mathfrak{p}_{a,-2d}(i)}^2\right)\omega\right) \\
&= 2\sum_{d=-\ceil{\log_2(a)/2}}^{\floor{n/2}} \sum_{i=1}^{\#\mathcal{P}_{a,-2d}} e\left( \overline{\mathfrak{p}_{a,-2d}(i)}^2 \omega\right) \sum_{k=0}^{2^{n-\norm{\mathcal{P}_{a,-2d}}}-1} e\left(2^{\norm{\mathcal{P}_{a,-2d}}}k \omega\right) \\
&= 2(e(2^n\omega)-1) \sum_{d=-\ceil{\log_2(a)/2}}^{\floor{n/2}} \sum_{i=1}^{\#\mathcal{P}_{a,-2d}}   \dfrac{e\left( \overline{\mathfrak{p}_{a,-2d}(i)}^2 \omega\right)}{e\left(2^{\norm{\mathcal{P}_{a,-2d}}}\omega \right)-1} 
\end{align*}
By the first claim of Lemma \ref{TreeBound},
\begin{align*}
\abs{\widetilde{C}_{2^n}(a,\omega,\mathbf{u})} &\leq 8a^3 \sum_{d=-\ceil{\log_2(a)/2}}^{\floor{n/2}} \dfrac{1}{\abs{\sin\left(\pi 2^{\norm{\mathcal{P}_{a,-2d}}+1}\omega \right)}} \\
&\leq 4\pi a^3 \sum_{d=-\ceil{\log_2(a)/2}}^{\floor{n/2}} \norm{ 2^{\norm{\mathcal{P}_{a,-2d}}+1}\omega }^{-1}_{\R/\Z}
\end{align*}
By the second claim of the same lemma,
\begin{align*}
\abs{\widetilde{C}_{2^n}(a,\omega,\mathbf{u})} &\leq 4\pi a^3 \sum_{d=-\ceil{\log_2(a)/2}}^{\floor{n/2}} \norm{ 2^{2\ceil{\log_2(a)} + 2d+1}\omega }^{-1}_{\R/\Z}.
\end{align*}
Now we recall a well-known consequence of Borel-Cantelli lemma.
\begin{lemma}[see \cite{baake2017averaging}, Lemma 6.2.6]
	\label{LemmaLowerBound}
	Let $\epsilon > 0$ and $q>1$. For almost every $x\in \R$ there exists $d_0(x)\in \N$ such that
	\[ 
	\norm{q^nx}_{\R/\Z} \geq \dfrac{1}{n^{1+\epsilon}},
	\] 
	holds for every $n\geq d_0(x)$.
\end{lemma}
Then, 
\begin{align*}
\abs{\widetilde{C}_{2^n}(a,\omega,\mathbf{u})} &= O_{a,\omega}\left( \sum^{\floor{n/2}}_{d=d_0(\omega)} (2\ceil{\log_2(a)} + 2d+1)^{1+\epsilon} \right)\\
&= O_{a,\omega}(n^{2+\epsilon})
\end{align*}

This finishes the proof for the case $t=1$. We have already seen the case $t=0$, which is the case of twisted Birkhoff sums in last section.\\

To finish this section, we will pass to uniform bounds for the twisted correlations for all points of the subshift. We relay on the classic prefix-suffix decomposition:
\begin{prop}[\cite{DumontThomas}, Lemma 1.6]\label{prefsuf}
	Let $\zeta$ be a primitive substitution with Perron-Frobenius eigenvalue $\theta$. For any $\mathbf{x}\in X_\zeta$, and $N\geq 1$ there exists $n\geq 1$ such that 
	\[
	\mathbf{x}_{[0,N-1]} = s_0\zeta(s_1)\ldots\zeta^{n}(s_n)\zeta^{n}(p_n)\ldots\zeta(p_1)p_0,
	\]
	where $p_i,s_i$ are respectively proper prefixes and suffixes (possibly empty) of words of $\{\zeta(a)|a\in\A \}$ and $p_n,s_n$ not both equal to the empty word. Moreover,
	there exists $C>0$ such that $|\log(N)-n\log(\theta)| \leq C$.
\end{prop}
\begin{cor}\label{TwistedCorrUnif}
	Let $C^{\pm}_{N}(a_1,\dots,a_t,\omega,\mathbf{x})$ be the twisted correlation over $\mathbf{x}\in X_{\zeta_{TM}}$. Then
	\begin{itemize}
		\item if $t$ is even, there exists $B>0$ independent of $\omega$ and $t$ such that for almost every $\omega$,
		\[
		C^{\pm}_{N}(a_1,\dots,a_t,\omega,\mathbf{x}) = O_{a_t,\omega}\left(\log(N)^{t+1}e^{ B\sqrt{\log(N)\log\log\log(N)}} \right).
		\]
		\item if $t$ is odd, for every $\epsilon>0$ and almost every $\omega$,
		\[
		C^{\pm}_{N}(a_1,\dots,a_t,\omega,\mathbf{x}) = O_{a_t,\omega}\left(\log(N)^{t+2+\epsilon}\right).
		\]
	\end{itemize}
\end{cor}
\begin{proof} Following the decomposition from Proposition \ref{prefsuf}, the bound follows simply from triangle inequality:
	\begin{align*}
	&\abs{C^{\pm}_{N}(a_1,\dots,a_t,\omega,\mathbf{x})} \\
	&\leq \sum_{j=0}^n \delta_{p_j,0}\abs{C^{\pm}_{\abs{\zeta_{TM}^j(p_j)}}(a_1,\dots,a_t,\omega,\mathbf{u})} + \delta_{p_j,1}\abs{C^{\pm}_{\abs{\zeta_{TM}^j(p_j)}}(a_1,\dots,a_t,\omega,\mathbf{v})}\\
	&+ \sum_{j=0}^n \delta_{s_j,0}\abs{C^{\pm}_{\abs{\zeta_{TM}^j(s_j)}}(a_1,\dots,a_t,\omega,\mathbf{u})} + \delta_{s_j,1}\abs{C^{\pm}_{\abs{\zeta_{TM}^j(s_j)}}(a_1,\dots,a_t,\omega,\mathbf{v})}\\
	&= \begin{cases}
	\sum_{j=0}^n O_{a_t,\omega}\left( j^t e^{ B\sqrt{j\log\log(j)}} \right) \quad &\text{ if } t \text{ is even,}\\
	\sum_{j=0}^n O_{a_t,\omega}\left( j^{t+1+\epsilon} \right)  &\text{ if } t \text{ is odd.}
	\end{cases}\\
	&\leq \begin{cases}
	(n+1)O_{a_t,\omega}\left( n^t e^{ B\sqrt{n\log\log(n)}} \right)\quad   & \text{ if } t \text{ is even,}\\
	(n+1)O_{a_t,\omega}\left( n^{t+1+\epsilon} \right) &\text{ if } t \text{ is odd.}
	\end{cases}
	\end{align*}
	By using the second claim of Proposition \ref{prefsuf}, we finish the proof of the corollary.
\end{proof}

\section{Bijective substitutions on two letters}
In this section we generalize results from Section 3 to general bijective substitutions on two letters. For the rest of the section let $q$ be an integer greater or equal to two.

\subsection{Top Lyapunov exponent} First let us see an arithmetic property for a general bijective substitution on two letters.
\begin{definition}
	We say a function $h:\N \longrightarrow \Z$ is \textit{strongly q-additive} if for all $n\geq 0$ and $j=0,\dots,q-1$
	\[
	h(qn+j) = h(n) + h(j). 
	\]
\end{definition}
\begin{remark}
	Note that a strongly $q$-additive function is completely determined by the values $h(1), \dots , h(q-1)$ (and $h(0)$ is necessarily equal to $0$). For more properties on (strongly) $q$-additive functions, see for example \cite{drmota2010analysis}.
\end{remark}
\begin{example}
	The sum-of-digits function $s_2$ is a strongly $2$-additive function.
\end{example}
From now on consider consider $\zeta$ a bijective substitution of length $q$ on $\A=\{0,1\}$. We suppose $\zeta(0)$ start with $0$, since otherwise we can work with $\zeta^2$ which is a bijective substitution of length $q^2$ on $\A$ and necessarily satisfies the latter condition. It will have associated the same subshift (and so, the same subshift factors) and the fact that
\[
\dfrac{\chi_{\zeta}^+}{\log(q)} = \dfrac{\chi_{\zeta^k}^+}{\log(q^k)} \quad \text{for all } k\geq1,
\]
justify this assumption to prove Theorem \ref{main}. Consider $\mathbf{u} = (u_n)_{n\geq0} = \zeta^\infty(0)$.
\begin{lemma}
	For $\zeta$ as above, the function $F = \mathds{1}_{[0]} - \mathds{1}_{[1]}$ defined on $X_\zeta$ satisfy
	\[
	F(T^{qn+j}\mathbf{u}) = F(T^n\mathbf{u})F(T^j\mathbf{u})
	\]
	(i.e., $F$ is strongly q-multiplicative, see \cite{drmota2010analysis}).
\end{lemma}
\begin{proof}
	Since $\mathbf{u} = \zeta(\mathbf{u})$ we have that for all $n\geq0$ and $0\leq j < q$, $u_{qn+j}$ = $\zeta(u_n)_j$. Let $H(n) = F(T^n\mathbf{u})$, then
	\begin{align*}
		H(qn+j) = \begin{cases} \: 1 &\text{if } u_{qn+j} = 0\\ \: -1 &\text{if } u_{qn+j} = 1 \end{cases}
		= \begin{cases} \: 1 &\text{if } \zeta(u_n)_j = 0\\ \: -1 &\text{if } \zeta(u_n)_j = 1 \end{cases}
	\end{align*}
	and
	\begin{align*}
	H(j) = \begin{cases} \: 1 &\text{if } u_{j} = 0\\ \: -1 &\text{if } u_{j} = 1 \end{cases}
	= \begin{cases} \: 1 &\text{if } \zeta(0)_j = 0\\ \: -1 &\text{if } \zeta(0)_j = 1 \end{cases}.
	\end{align*}
	If $u_n = 0$, then $H(n) = 1$ and
	\[
	H(qn+j) = \begin{cases} \: 1 &\text{if } \zeta(0)_j = 0\\ \: -1 &\text{if } \zeta(0)_j = 1 \end{cases} = H(j) = H(n)H(j).
	\]
	If $u_n = 1$, then $H(n) = -1$ and
	\[
	H(qn+j) = \begin{cases} \: 1 &\text{if } \zeta(1)_j = 0\\ \: -1 &\text{if } \zeta(1)_j = 1 \end{cases} = \begin{cases} \: 1 &\text{if } \zeta(0)_j = 1\\ \: -1 &\text{if } \zeta(0)_j = 0 \end{cases}= -H(j) = H(n)H(j),
	\]
	where we used the bijectivity of the substitution in the second equality.
\end{proof}
\begin{cor}\label{arithmeticBij}
	There is a strongly $q$-additive function $h:\N \longrightarrow \Z$ such that
	\[
	F(T^n\mathbf{u}) = (-1)^{h(n)}
	\]
\end{cor}

\begin{proof}
	It is clear we can define a function $\tilde{h}:\N \longrightarrow \Z/2\Z$ that satisfy the strong $q$-additive property. To pass to a strongly $q$-additive function $h:\N \longrightarrow \Z$ such that for all $n\geq1$ we have $h(n) \equiv \tilde{h}(n) \;(\text{mod } 2)$, we just have to make any choice of elements in $\{0,\dots, 2(q-1)\}$ such that $h(j) \equiv \tilde{h}(j) \: (\text{mod } 2)$ for $j=1,\dots,q-1$ and define the rest the terms $h(n)$ by induction.
\end{proof}

\subsection{Twisted correlations} As before, we will consider the twisted correlations over a fixed point of the substitution and for the function $F = \mathds{1}_{[0]} - \mathds{1}_{[1]}$. By Corollary \ref{arithmeticBij}, consider $h:\N \longrightarrow \Z$ a strongly $q$-additive function such that
\[
F(T^n\mathbf{u}) = (-1)^{h(n)}.
\]
Therefore we will study the sum
\[
C^{\pm}_{q^n}(a_1,\dots,a_t,\omega, \mathbf{u}) = \sum_{k=0}^{q^n-1} (-1)^{h(k)+ h(k+a_1)+\dots+h(k+a_t)}e(k\omega).
\]
As in the case of Thue-Morse, we have the next result for the growth of twisted correlations: let $I = \chi^+_{\zeta}$ be the top Lyapunov exponent of the spectral cocycle.
\begin{theorem}\label{TwistCorrGen}
	Let $C^{\pm}_{q^n}(a_1,\dots,a_t,\omega,\mathbf{u})$ be the twisted correlation defined in Definition \ref{defTwsited}, for $F = \mathds{1}_{[0]}-\mathds{1}_{[1]}$. Then,
	\begin{itemize}
		\item if $t$ is even, there exists $B>0$ independent of $\omega$ and $t$ such that for almost every $\omega$,
		\[
		C^{\pm}_{q^n}(a_1,\dots,a_t,\omega,\mathbf{u}) = O_{a_t,\omega}\left(n^t e^{nI + B\sqrt{n\log\log(n)} } \right).
		\]
		\item if $t$ is odd, for every $\epsilon>0$ and almost every $\omega$,
		\[
		C^{\pm}_{q^n}(a_1,\dots,a_t,\omega,\mathbf{u}) = O_{a_t,\omega}(n^{t+1+\epsilon}) .
		\]
	\end{itemize}
\end{theorem}
We will prove this result in the Appendix. Of course, we may generalize to an arbitrary point of the subshift $X_\zeta$ using the prefix-suffix decomposition.
\begin{cor}\label{TwistedCorrUnifGen}
	Let $C^{\pm}_{N}(a_1,\dots,a_t,\omega,\mathbf{x})$ be the twisted correlation over $\mathbf{x}\in X_{\zeta}$. Then
	\begin{itemize}
		\item if $t$ is even, there exists $B>0$ independent of $\omega$ and $t$ such that for almost every $\omega$,
		\[
		C^{\pm}_{N}(a_1,\dots,a_t,\omega,\mathbf{x}) = O_{a_t,\omega}\left(N^{I/\log(q)}\log(N)^{t+1} e^{B\sqrt{\log(N)\log\log\log(N)} } \right).
		\]
		\item if $t$ is odd, for every $\epsilon>0$ and almost every $\omega$,
		\[
		C^{\pm}_{N}(a_1,\dots,a_t,\omega,\mathbf{x}) = O_{a_t,\omega}(\log(N)^{t+2+\epsilon}) .
		\]
	\end{itemize}
\end{cor}
\begin{proof}
	The proof is the same of Corollary \ref{TwistedCorrUnif}.
\end{proof}

\section{Top Lyapunov exponents of factors}
In this section we will prove Theorem \ref{main}. The proof relies on the bounds for the twisted correlations found in subsection \ref{twistsec}: we will see that the twisted Birkhoff sums of the factor subshift, correspond to linear combinations of twisted correlations on the factorized system.\\

The sketch of the proof has already been shown in the background section, but here is writen down explicitly and rigorously. Let $(X_{\sigma},T)$ be an aperiodic topological factor of $(X_{\zeta},T)$ arising from the aperiodic substitution $\sigma$, where $\zeta$ is a bijective substitution on the alphabet $\A = \{0,1\}$ of constant length $q\geq2$. By Theorem \ref{DuLe}, the factor map $\pi: X_{\zeta}\longrightarrow X_{\sigma}$ is defined by a map $\widehat{\pi}$ of radius $r$, where $r$ is a computable constant depending on $X_{\sigma}$. Let $I = \chi^+_\zeta$ be the top Lyapunov exponent of the spectral cocycle for $\zeta$, being almost everywhere constant since $\zeta$ is of constant length. Consider a point $\mathbf{w}$ such that at position zero it has the block $\sigma^n(a)$, for a letter $a\in \A$. Choose any $\mathbf{x}\in \pi^{-1}(\{\mathbf{w}\})$. Then, setting $f = \mathds{1}_{[b]}$, we have
\begin{align*}
\abs{\mathscr{C}_{\sigma}(\omega,n)(a,&b)} = \abs{S^f_{\abs{\sigma^n(a)}}(\omega,\mathbf{w})}\\
&\leq \sum_{w: \widehat{\pi}(w) = b} \left|  \sum_{k=0}^{\abs{\sigma^n(a)}-1} \mathds{1}_{[w_{-r}]}(T^{k-r}\mathbf{x})\cdot\ldots\cdot\mathds{1}_{[w_{r}]}(T^{k+r}\mathbf{x}) e(k\omega)\right| \\
&\leq \sum_{w: \widehat{\pi}(w) = b} \left|  \sum_{k=0}^{\abs{\sigma^n(a)}-(r+1)} \mathds{1}_{[w_{-r}]}(T^{k}\mathbf{x})\cdot\ldots\cdot\mathds{1}_{[w_{r}]}(T^{k+2r}\mathbf{x}) e(k\omega)\right| + \Delta_r.
\end{align*}

In the last sum, there are terms depending on negative coordinates of $\mathbf{x}$, but at most on the first $r$ negative coordinates, which we leave apart in the term $\Delta_r$ since they make no difference in the asymptotic as $n$ goes to infinity.\\

Consider the next simple relations for $F(\mathbf{x}) = \mathds{1}_{[0]}(\mathbf{x}) - \mathds{1}_{[1]}(\mathbf{x})$:
\begin{align*}
\mathds{1}_{[0]}(\mathbf{x}) &= \dfrac{1}{2}\left[ 1 + F(\mathbf{x}) \right],\\
\mathds{1}_{[1]}(\mathbf{x}) &= \dfrac{1}{2}\left[ 1 - F(\mathbf{x})\right].
\end{align*}
By definition $C_N^{\pm}(a_1,\dots,a_t,\omega,\mathbf{x}) = C^F_N(a_1,\dots,a_t,\omega,\mathbf{x})$. In the last inequality above we replace each factor $\mathds{1}_{[0]}(\mathbf{x}), \mathds{1}_{[1]}(\mathbf{x})$ by $\dfrac{1}{2}\left[ 1 + F(\mathbf{x}) \right]$, $\dfrac{1}{2}\left[ 1 - F(\mathbf{x})\right]$ respectively, and develop the product, yielding a sum of twisted correlations with parameters $a_i$ belonging to non-empty subsets of $\{0,\dots,2r\}$, of the function $F$ (note the obvious abuse of notation for the parameters in the twisted correlations below and the addition of one of these terms into $\Delta_r$):
\[
\abs{\mathscr{C}_{\sigma}(\omega,n)(a,b)} \leq  \sum_{w: \widehat{\pi}(w) = b} \, \sum_{\substack{\mathcal{S} \subset \{0,\dots,2r\} \\ \mathcal{S} \neq \emptyset } }  \left| C^F_{\abs{\sigma^n(a)}-(r+1)}(\mathcal{S},\omega,\mathbf{x}) \right| + \widetilde{\Delta_{r,b}}.
\]
We may estimate each twisted correlation using Corollary \ref{TwistedCorrUnifGen}. We claim that for any $\epsilon > 0$ there exists a constant $C_{b,r,\omega}>0$ depending on $b,r,\omega$ such that
\begin{align*}
	\abs{\mathscr{C}_{\sigma}(\omega,n)(a,b)} &\leq C_{b,r,\omega} {\abs{\sigma^n(a)}}^{I/\log(q) + \epsilon} \\
	&\leq C_{b,r,\omega} (\norm{(M^T_\sigma)^n}_\infty)^{I/\log(q) + \epsilon}
\end{align*}
The second inequality is obvious by definition of the substitution matrix and the norm $\norm{\cdot}_{\infty}$. For the first one, note that by Corollary \ref{TwistedCorrUnifGen} we have that for $N = \abs{\sigma^n(a)}-(r+1)$
\[
\left| C^F_{N}(\mathcal{S},\omega,\mathbf{x}) \right| =  O_{a_t,\omega}\left(N^{I/\log(q)}\log(N)^{2r+2} e^{B\sqrt{\log(N)\log\log\log(N)}} \right),
\]
since the bound in Corollary \ref{TwistedCorrUnifGen} in the odd case dominates the one of the even case. Since for any $\epsilon>0$ we have, $\log(N)^{2r+2} = O_r(N^{\epsilon})$ and $e^{B\sqrt{\log(N)\log\log\log(N)}} = O(N^{\epsilon})$ (for our purposes, $B$ is a universal constant). Then
\begin{align*}
	N^{I/\log(q)}\log(N)^{2r+2} e^{B\sqrt{\log(N)\log\log\log(N)} } &= O_r(N^{I/\log(q)+\epsilon})\\
	& = O_r\left( {\abs{\sigma^n(a)}}^{I/\log(q) + \epsilon}  \right),
\end{align*}
which proves the claim. \\

This implies
\begin{align*}
\chi^{+}_{\sigma}(\omega) &= \limsup_{n\to\infty} \dfrac{1}{n}\log(\abs{\mathscr{C}_{\sigma}(\omega,n)(a,b)})\\ &\leq ( I/\log(q) + \epsilon)\log\left(\limsup_{n\to\infty} \, \norm{(M^T_\sigma)^n}^{1/n}_\infty\right) \\
&= \log\left(\rho(M^T_\sigma)\right)( I/\log(q) + \epsilon),
\end{align*}
i.e.,
\[
\dfrac{\chi^{+}_{\sigma}(\omega)}{\log(\rho(M_\sigma))} \leq \dfrac{I}{\log(q)} + \epsilon
\]
(note that $M^T_\sigma$ has the same spectral radius as $M_\sigma$ since they have the same eigenvalues). Since $\epsilon>0$ is arbitrary, we obtain the conclusion of Theorem \ref{main}.

\section{Appendix}
\subsection{Preliminary lemmata} Let $h:\N \longrightarrow \Z$ be a strongly $q$-additive function. As in the case of $s_2$, we consider for $a\in \N$ and $d\in\Z$ the set $\mathcal{S}_{a,d}$ of positive integer solutions $n$ to the equation
\[
h(n+a)-h(n) = d.
\]
We have the same structure for the solutions as shown by the next lemma. The proof is the same as in Lemma 2.1.1 in \cite{emme2017asymptotic}. We recall the notation introduced in subsection \ref{Secsolutions}: let $k$ be a nonnegative integer. Let $\underline{k}_q$ be the word associated to the digits of $k$ in base $q$, i.e., if $k = k_0 + \dots + k_{n-1}q^{n-1}$ then, $\underline{k}_q = k_0\dots k_{n-1}\in\{0,1, \dots,q-1\}^*$. Similarly, for a word $w=w_0\dots w_{n-1} \in  \{0,1,\dots,q-1\}^*$, let $\overline{w}^q$ be the number $w_0 + \dots + w_{n-1}q^{n-1}$.
\begin{lemma}\label{lemmaWords}
	There exists a (possibly empty) finite set of words
	\[
	\mathcal{P}_{a,d} = \{\mathfrak{p}^d_a(1),\dots,\mathfrak{p}^d_a(s)\} \subset \{0,1,\dots,q-1\}^*, 
	\]
	such that
	\[
	k\in \mathcal{S}_{a,d} \iff \underline{k}_q \in \bigcup_{i=1}^s \: [\mathfrak{p}^d_a(i)],
	\]
\end{lemma}
\begin{proof}
	Let us proceed by induction on $a$: if $a = 1$ and $k \in \mathcal{S}_{1,d}$ we have two cases:
	\begin{itemize}
		\item $k = ql + r$ for some $l\geq 0$ and $0\leq r < q-1$. Then $h(k+1) = h(l) + h(r+1)$ and since $h(k) = h(l) + h(r)$, the equation is equivalent to
		\[
		h(r+1) - h(r) = d.
		\]
		There is only a finite number of prefixes (least significant digits), in fact of length 1, that satisfy this equation.\\
		\item $k = lq^{p+1} + rq^p + (q-1)q^{p-1} + \dots +(q-1)$ for some $p\geq1$ and $0\leq r < q-1$. Then, $h(k+1) = h(l) + h(r+1)$ and since $h(k) = h(l) + h(r) + ph(q-1)$, the equation is equivalent to
		\[
		h(r+1) - h(r) - ph(q-1) = d
		\] 
		In particular, $p$ is bounded for fixed $d$ and again there is only a finite number of prefixes that satisfy this equation.
	\end{itemize}
	Suppose this is true for $a\geq 1$ and let $m$ be an integer between $a+1$ and $qa$. Let $k\equiv j \: (\text{mod } q)$, then 
	\[
	h(k) = h\left(\dfrac{k-j}{q}\right) + h(j).
	\]
	Note that $a+1\leq m+j =: ql+r \leq qa+j$, and so $l\leq a$. Then 
	\[
	h(k+m) = h\left( \dfrac{k-j}{q} + l\right) + h(r).
	\]
	That is, $k = qk' + j$ and $k'\in \mathcal{S}_{l,d+h(j)-h(r)}$.
\end{proof}
An analog of Lemma \ref{TreeBound} holds in this broader case. We will keep using $\mathcal{P}_{a,d}$ to denote the set of words in $\mathcal{P}_{a,d}$ (if it is non-empty) after adding suffixes in such a way the length of every word is of equal length to the longest one.
\begin{lemma}\label{TreeBoundGeneral}
	For all $a \in \N$ and $d\in\Z$, there exist positive constants $\mathfrak{C}_a$ and $\mathfrak{D}_a$ such that the set $\mathcal{P}_{a,d}$, satisfies
	\begin{itemize}
		\item $\#\mathcal{P}_{a,d} \leq \mathfrak{C}_a$.
		\item $\norm{\mathcal{P}_{a,d}} = \mathfrak{D}_a + \abs{d}$.
	\end{itemize}
\end{lemma}
\begin{proof}
	For the first claim, note that in the proof of Lemma \ref{lemmaWords} if $a=1$ there is a fixed number of prefixes (or least significant digits) for the solutions of the equations (considering them as words). For $m\geq1$ between $a+1$ and $qa$, we reduce the problem by dividing $m$ by $q$. Repeating this procedure at most the length of $\underline{a}_q$ times, we obtain a finite number of prefixes and the number of them only depends on $a$, that is, we bound $\#\mathcal{P}_{a,d}$ by a positive constant $\mathfrak{C}_a$.\\
	
	For the second claim, note that in the proof of Lemma \ref{lemmaWords} if $a=1$ the length of the prefixes of the solutions is bounded by some fixed constant plus the absolute value of $d$. For an integer $m\geq1$ between $a+1$ and $qa$, we divide again by $q$ to reduce the problem, and we obtain that the length of the possible prefixes is bounded by a contant depending on $a$ plus $\abs{d}$.\\
	
	Finally, to obtain the equality of the second claim, we simply complete the words of $\mathcal{P}_{a,d}$ with all possible suffixes, to make all words of the same length. Since the possible suffixes only depends on $a$, we may simply redefine $\mathfrak{C}_a$ in such a way that the cardinal $\mathcal{P}_{a,d}$ is bounded by this constant.	
\end{proof}
Both Lemma \ref{lemmaWords} and Lemma \ref{TreeBoundGeneral} have a multiple versions.
\begin{cor}\label{IntersectionWords}
	Let $1\leq a_1<\ldots<a_t$ be positive integers and $d_1, \ldots d_t\in\Z$. Then
	\[
	k\in\bigcap_j\mathcal{S}_{a_j,d_j} \iff \underline{k}_q \in \bigcup_i \: [\mathfrak{p}^{d_1,\dots,d_t}_{a_1,\dots,a_t }(i)],
	\]
	for some words $\mathfrak{p}^{d_1,\dots,d_t}_{a_1,\dots,a_t }(i) \in \bigcap_j \mathcal{P}_{a_j,d_j}=: \mathcal{P}^{d_1,\dots,d_t}_{a_1,\dots,a_t}$, where $i = 1,\dots,\#\mathcal{P}^{d_1,\dots,d_t}_{a_1,\dots,a_t}$ if this set is non-empty.
\end{cor}

\begin{cor}\label{IntersectionBounds}
	By completing shorter words of $\mathcal{P}^{d_1,\dots,d_t}_{a_1,\dots,a_t}$ (if this set is non-empty) there exist positive constants $\mathfrak{C}_{a_t},\mathfrak{D}_{a_t}$, such that
	\begin{itemize}
		\item $\#\mathcal{P}^{d_1,\dots,d_t}_{a_1,\dots,a_t} \leq \mathfrak{C}_{a_t}$.\\
		\item $\norm{\mathcal{P}^{d_1,\dots,d_t}_{a_1,\dots,a_t}} =  \mathfrak{D}_{a_t} + \abs{d_{j^*}}$, for some $1\leq j^*\leq t$.
	\end{itemize}
\end{cor}
Both corollaries can be checked rapidly since all of these words correspond to elements of some $\mathcal{P}_{a_j,d_j}$ which has the longest words among all these sets, or it is empty for the choice of $a_1,\dots,a_t,d_1,\dots,d_t$.

\subsection{Upper bounds for twisted correlations}
Let us begin with the case when $t$ is odd. First note that the equation $h(k+a)-h(k) = d$ may have solutions in $\{0,\dots,q^n-1\}$ only if $d\in\{-2qn,\dots,h(a)\}$. For every $j=1,\dots,t$, consider $\mathcal{D}_j = \{-2qn,\dots, h(a_j)\}$ and the decomposition $\mathcal{D}_j = \mathcal{D}_j^0 \cup \mathcal{D}_j^1$ in even and odd elements respectively.\\

Since the geometric series of the exponentials is bounded by a constant depending on $\omega$, we consider equivalently $\widetilde{C}_{q^n}(a_1,\dots,a_t,\omega,\mathbf{u})= C^\pm_{q^n}(a_1,\dots,a_t,\omega,\mathbf{u}) + \sum_{k=0}^{q^n-1} e(k\omega)$. Note also that since $t$ is odd, we can change the exponent of $-1$ from $h(k) + h(k+a_1) + \dots + h(k+a_t)$ to 
\[
h(k+a_1)-h(k) + h(k+a_2)-h(k) \dots + h(k+a_t)-h(k).
\]
Using these facts we have
\begin{align*}
\widetilde{C}_{q^n}(a_1,\dots,a_t,\omega,\mathbf{u}) &= \sum_{k=0}^{q^n-1} \left(1+(-1)^{h(k) + h(k+a_1) + \dots + h(k+a_t)} \right) e(k\omega) \\
&= \sum_{k=0}^{q^n-1} \left(1+\prod_{j=1}^t(-1)^{h(k+a_j)-h(k)} \right) e(k\omega) \\
&= 2\sum_{\substack{d_j \in \mathcal{D}^{c_j}_j \\ c_1+\dots+c_t\equiv_2 0 }} \, \sum_{\substack{ k\in\cap_j\mathcal{S}_{a_j,d_j} \\ 0\leq k\leq q^n-1}} e(k\omega).
\end{align*}
The first sum in the last equality sums over all elements of each $\mathcal{D}^{c_1}_1,\dots,\mathcal{D}^{c_t}_t$, for all combinations of variables $c_1,\dots,c_t \in \{0,1\}$ satisfying $c_1+\dots+c_t\equiv_2 0$.\\

Applying Corollary \ref{IntersectionWords} and via simple calculations involving geometric series, we obtain
\begin{align*}
&\widetilde{C}_{q^n}(a_1,\dots,a_t,\omega,\mathbf{u})\\ 
&= 2\sum_{\substack{d_j \in \mathcal{D}^{c_j}_j \\ c_1+\dots+c_t\equiv_2 0 }} \, \sum_{\substack{ k\in\cap_j\mathcal{S}_{a_j,d_j} \\ 0\leq k\leq q^n-1}} e(k\omega). \\
&= 2\sum_{\substack{d_j \in \mathcal{D}^{c_j}_j \\ c_1+\dots+c_t\equiv_2 0 }} \, \sum_{i=1}^{\#\mathcal{P}^{d_1,\dots,d_t}_{a_1,\dots,a_t}} \:\: \sum_{k=0}^{q^{n-\norm{\mathcal{P}^{d_1,\dots,d_t}_{a_1,\dots,a_t}}}-1} e\left(\left(q^{\norm{\mathcal{P}^{d_1,\dots,d_t}_{a_1,\dots,a_t}}}k + \overline{\mathfrak{p}^{d_1,\dots,d_t}_{a_1,\dots,a_t }(i)}^q\right)\omega\right) \\
&= 2(e\left(q^{n}\omega \right)-1) \sum_{\substack{d_j \in \mathcal{D}^{c_j}_j \\ c_1+\dots+c_t\equiv_2 0 }} \, \sum_{i=1}^{\#\mathcal{P}^{d_1,\dots,d_t}_{a_1,\dots,a_t}}  \dfrac{ e\left( \overline{\mathfrak{p}^{d_1,\dots,d_t}_{a_1,\dots,a_t }(i)}^q\omega\right)}{e\left(q^{\norm{\mathcal{P}^{d_1,\dots,d_t}_{a_1,\dots,a_t}}}\omega \right)-1}.
\end{align*}
Taking the absolute value and by Corollary \ref{IntersectionBounds},
\begin{align*}
\abs{\widetilde{C}_{q^n}(a_1,\dots,a_t,\omega,\mathbf{u})}
&\leq 2\mathfrak{C}_{a_t} \sum_{\substack{d_j \in \mathcal{D}^{c_j}_j \\ c_1+\dots+c_t\equiv_2 0 }} \dfrac{1}{\abs{\sin\left(\pi q^{\norm{\mathcal{P}^{d_1,\dots,d_t}_{a_1,\dots,a_t}}+1 }\omega \right)}} \\
&\leq \pi \mathfrak{C}_{a_t} \sum_{\substack{d_j \in \mathcal{D}^{c_j}_j \\ c_1+\dots+c_t\equiv_2 0 }} \norm{ q^{\norm{\mathcal{P}^{d_1,\dots,d_t}_{a_1,\dots,a_t}}+1}\omega }^{-1}_{\R/\Z}.
\end{align*}

Since $\norm{\mathcal{P}^{d_1,\dots,d_t}_{a_1,\dots,a_t}} = \norm{\mathcal{P}_{a_{j^*},d_{j^*}}} = \mathfrak{D}_{a_t} + \abs{d_{j^*}}$ for some $j^*$, we obtain
\[
\abs{\widetilde{C}_{q^n}(a_1,\dots,a_t,\omega)} \leq \pi \mathfrak{C}_{a_t} \sum_{\substack{d_j \in \mathcal{D}^{c_j}_j \\ c_1+\dots+c_t\equiv_2 0 } } \norm{q^{\mathfrak{D}_{a_t} + \abs{d_{j^*}}+1} \omega}^{-1}_{\R/\Z}.
\]
To estimate each term in this sum we argue as in subsection \ref{twistsec}: by Lemma \ref{LemmaLowerBound}, for almost every $\omega$ there exists $d_0(\omega)$ such that
\[
\norm{q^n\omega}_{\R/\Z} \geq \dfrac{1}{n^{1+\epsilon}},
\]
for all $n\geq d_0(\omega)$. In particular, 
\begin{align*}
\abs{\widetilde{C}_{q^n}(a_1,\dots,a_t,\omega,\mathbf{u})} &\leq O_{a_t,\omega}(1) + \mathfrak{C}_{a_t} \sum_{l=1}^t \sum_{\substack{d_j \in \mathcal{D}^{c_j}_j \\ c_1+\dots+c_t\equiv_2 0 \\ j^* = l \\ \abs{d_{j^*}} \geq d_0(\omega) } } \left(\mathfrak{D}_{a_t} + \abs{d_{l}}+1\right)^{1+\epsilon} \\
&= O_{a_t,\omega}(n^{t+1+\epsilon}).
\end{align*}

The even case is worked in a similar way: consider 
\[
\widehat{C}_{q^n}(a_1,\dots,a_t,\omega,\mathbf{u}) = C^{\pm}_{q^n}(a_1,\dots,a_t,\omega,\mathbf{u}) + \sum_{k=0}^{q^n-1} (-1)^{h(k)}e(k\omega).
\]
Since $\sum_{k=0}^{q^n-1} (-1)^{h(k)}e(k\omega) = O_{\omega}\left(e^{nI+B\sqrt{n\log\log(n)}}\right)$ almost surely by Theorem \ref{deviation} and we want a weaker estimate, it is enough to prove the required upper bound for $\widehat{C}_{q^n}(a_1,\dots,a_t,\omega,\mathbf{u})$. Note also that since $t$ is even, we can change the exponent of $-1$ from $h(k+a_1) + \dots + h(k+a_t)$ to 
\[
h(k+a_1)-h(k) + h(k+a_2)-h(k) \dots + h(k+a_t)-h(k),
\]
to justify the third equality below (and also that $(-1)^{h(k)} = e(h(k)/2)$).
\begin{align*}
&\widehat{C}_{q^n}(a_1,\dots,a_t,\omega,\mathbf{u}) \\
&= \sum_{k=0}^{q^n-1} \left((-1)^{h(k)}+(-1)^{h(k) + h(k+a_1) + \dots + h(k+a_t)} \right) e(k\omega) \\
&= \sum_{k=0}^{q^n-1} (-1)^{h(k)} \left( 1+(-1)^{h(k+a_1) + \dots + h(k+a_t)}  \right) e(k\omega) \\
&= \sum_{k=0}^{q^n-1} \left(1+\prod_{j=1}^t(-1)^{h(k+a_j)-h(k)} \right) e\left(\dfrac{h(k)}{2}+k\omega\right) \\
&= 2\sum_{\substack{d_j \in \mathcal{D}^{c_j}_j \\ c_1+\dots+c_t\equiv_2 0 }} \, \sum_{\substack{ k\in\cap_j\mathcal{S}_{a_j,d_j} \\ 0\leq k\leq q^n-1}} e\left(\dfrac{h(k)}{2}+k\omega\right) \\
&= 2\sum_{\substack{d_j \in \mathcal{D}^{c_j}_j \\ c_1+\dots+c_t\equiv_2 0 }} \, \sum_{i=1}^{\#\mathcal{P}^{d_1,\dots,d_t}_{a_1,\dots,a_t}} \:\: \sum_{k=0}^{q^{n-\norm{\mathcal{P}^{d_1,\dots,d_t}_{a_1,\dots,a_t}}}-1} \mathfrak{A}_k, \\
\end{align*}
where
\[
\mathfrak{A}_k = e\left(\dfrac{h\left(q^{\norm{\mathcal{P}^{d_1,\dots,d_t}_{a_1,\dots,a_t}}}k + \overline{\mathfrak{p}^{d_1,\dots,d_t}_{a_1,\dots,a_t }(i)}^q\right)}{2} + \left(q^{\norm{\mathcal{P}^{d_1,\dots,d_t}_{a_1,\dots,a_t}}}k + \overline{\mathfrak{p}^{d_1,\dots,d_t}_{a_1,\dots,a_t }(i)}^q\right)\omega\right).
\]

Let us factorize the terms independent of $k$ in the exponentials $\mathfrak{A}_k$, and continue from the last line.
\begin{align*}
&\widehat{C}_{q^n}(a_1,\dots,a_t,\omega,\mathbf{u}) \\
&= 2\sum_{\substack{d_j \in \mathcal{D}^{c_j}_j \\ c_1+\dots+c_t\equiv_2 0 }} \sum_{i=1}^{\#\mathcal{P}^{d_1,\dots,d_t}_{a_1,\dots,a_t}} \underbrace{e\left(\dfrac{h\left(\overline{\mathfrak{p}^{d_1,\dots,d_t}_{a_1,\dots,a_t }(i)}^q\right)}{2} + \overline{\mathfrak{p}^{d_1,\dots,d_t}_{a_1,\dots,a_t }(i)}^q\omega\right) }_{=\mathfrak{B}^{d_1,\dots,d_t}_{a_1,\dots,a_t}(i)} \\
& \qquad \qquad \qquad \qquad \qquad \sum_{k=0}^{q^{n-\norm{\mathcal{P}^{d_1,\dots,d_t}_{a_1,\dots,a_t}}}-1} e\left(\dfrac{h\left(q^{\norm{\mathcal{P}^{d_1,\dots,d_t}_{a_1,\dots,a_t}}}k\right)}{2} + k q^{\norm{\mathcal{P}^{d_1,\dots,d_t}_{a_1,\dots,a_t}}} \omega \right) \\
&= 2\sum_{\substack{d_j \in \mathcal{D}^{c_j}_j \\ c_1+\dots+c_t\equiv_2 0 }} \sum_{i=1}^{\#\mathcal{P}^{d_1,\dots,d_t}_{a_1,\dots,a_t}} \mathfrak{B}^{d_1,\dots,d_t}_{a_1,\dots,a_t}(i) \sum_{k=0}^{q^{n-\norm{\mathcal{P}^{d_1,\dots,d_t}_{a_1,\dots,a_t}}}-1} e\left(\dfrac{h\left(k\right)}{2} + kq^{\norm{\mathcal{P}^{d_1,\dots,d_t}_{a_1,\dots,a_t}}}\omega\right)\\
&= 2\sum_{\substack{d_j \in \mathcal{D}^{c_j}_j \\ c_1+\dots+c_t\equiv_2 0 }} \sum_{i=1}^{\#\mathcal{P}^{d_1,\dots,d_t}_{a_1,\dots,a_t}} \mathfrak{B}^{d_1,\dots,d_t}_{a_1,\dots,a_t}(i) \prod_{j=0}^{n-\norm{\mathcal{P}^{d_1,\dots,d_t}_{a_1,\dots,a_t}}-1} P(\omega q^{\norm{\mathcal{P}^{d_1,\dots,d_t}_{a_1,\dots,a_t}}+j})
\end{align*}
(we use Lemma \ref{product} in the last equality). Taking the absolute value,
\begin{align*}
&\abs{\widehat{C}_{q^n}(a_1,\dots,a_t,\omega,\mathbf{u})} \\
&\leq 2 \sum_{\substack{d_j \in \mathcal{D}^{c_j}_j \\ c_1+\dots+c_t\equiv_2 0 }} \sum_{i=1}^{\#\mathcal{P}^{d_1,\dots,d_t}_{a_1,\dots,a_t}}  \prod_{j=0}^{n-\norm{\mathcal{P}^{d_1,\dots,d_t}_{a_1,\dots,a_t}}-1} \left| P (\omega q^{\norm{\mathcal{P}^{d_1,\dots,d_t}_{a_1,\dots,a_t}}+j}) \right| \\
&\leq 2\mathfrak{C}_{a_t} \sum_{\substack{d_j \in \mathcal{D}^{c_j}_j \\ c_1+\dots+c_t\equiv_2 0 }} \prod_{j=\norm{\mathcal{P}^{d_1,\dots,d_t}_{a_1,\dots,a_t}}}^{n-1} \abs{P(\omega q^j)} \\
&= 2\mathfrak{C}_{a_t} \sum_{l=1}^t \sum_{\substack{d_j \in \mathcal{D}^{c_j}_j \\ c_1+\dots+c_t\equiv_2 0 \\ j^* =\: l } } \prod_{j=\mathfrak{D}_{a_t} + \abs{d_l}}^{n-1} \abs{P(\omega q^j)}.
\end{align*}
By Theorem \ref{deviation} there exists a universal $B>0$ such that for almost all $\omega$,
\[
\prod_{j=0}^{n-1} \left|\dfrac{P(\omega q^j)}{e^I}\right| \leq e^{ B\sqrt{n\log\log(n)}},
\]
for all $n\geq n_0(\omega)$, for some $n_0(\omega)$. In particular,
\[
\prod_{j=l}^{n-1} \left|P(\omega q^j)\right| = e^{(n-l)I}\prod_{j=0}^{n-1} \left|\dfrac{P(\omega q^j)}{e^I}\right| \leq e^{nI + B\sqrt{n\log\log(n)}}.
\]
This is, for all $n\geq0$
\[
\prod_{j=l}^{n-1} \left|P(\omega q^j)\right| \leq O_{\omega}\left(e^{nI + B\sqrt{n\log\log(n)}} \right).
\]
This finishes the proof of Theorem \ref{TwistCorr}.
\bibliographystyle{amsalpha}
\bibliography{bibliografia.bib}

\providecommand{\bysame}{\leavevmode\hbox to3em{\hrulefill}\thinspace}
\providecommand{\MR}{\relax\ifhmode\unskip\space\fi MR }
\providecommand{\MRhref}[2]{%
  \href{http://www.ams.org/mathscinet-getitem?mr=#1}{#2}
}
\providecommand{\href}[2]{#2}
\begin{thebibliography}{DGM12}

\bibitem[AFS23]{avila2023quantitative}
Artur Avila, Giovanni Forni, and Pedram Safaee, \emph{Quantitative weak mixing
  for interval exchange transformations}, Geometric and Functional Analysis
  \textbf{33} (2023), no.~1, 1--56.

\bibitem[BC23]{baake2023correlations}
Michael Baake and Michael Coons, \emph{Correlations of the {T}hue--{M}orse
  sequence}, Indagationes Mathematicae (2023).

\bibitem[BCM20]{BaakeBinary}
Michael Baake, Michael Coons, and Neil Ma{\~{n}}ibo, \emph{Binary
  constant-length substitutions and {M}ahler measures of borwein polynomials},
  From Analysis to Visualization (Cham) (David~H. Bailey, Naomi~Simone Borwein,
  Richard~P. Brent, Regina~S. Burachik, Judy-anne~Heather Osborn, Brailey Sims,
  and Qiji~J. Zhu, eds.), Springer International Publishing, 2020,
  pp.~303--322.

\bibitem[BGM19]{baake2019renormalisation}
Michael Baake, Franz G{\"a}hler, and Neil Manibo, \emph{Renormalisation of pair
  correlation measures for primitive inflation rules and absence of absolutely
  continuous diffraction}, Communications in Mathematical Physics \textbf{370}
  (2019), no.~2, 591--635.

\bibitem[BHL17]{baake2017averaging}
Michael Baake, Alan Haynes, and Daniel Lenz, \emph{Averaging almost periodic
  functions along exponential sequences. {I}n: {B}aake, {M}., {G}rimm, {U}.
  (eds.): Aperiodic order, vol. 2: Crystallography and almost periodicity, pp.
  343– 362}, Cambridge University Press, Cambridge, 2017.

\bibitem[BS14]{bufetov2014modulus}
Alexander~I. Bufetov and Boris Solomyak, \emph{On the modulus of continuity for
  spectral measures in substitution dynamics}, Advances in Mathematics
  \textbf{260} (2014), 84--129.

\bibitem[BS18]{bufetov2018holder}
\bysame, \emph{The {H}{\"o}lder property for the spectrum of translation flows
  in genus two}, Israel Journal of Mathematics \textbf{223} (2018), no.~1,
  205--259.

\bibitem[BS19]{berlinkov2019singular}
Artemi Berlinkov and Boris Solomyak, \emph{Singular substitutions of constant
  length}, Ergodic Theory and Dynamical Systems \textbf{39} (2019), no.~9,
  2384--2402.

\bibitem[BS20]{bufetov2020spectral}
Alexander~I. Bufetov and Boris Solomyak, \emph{A spectral cocycle for
  substitution systems and translation flows}, Journal d'analyse
  math{\'e}matique \textbf{141} (2020), no.~1, 165--205.

\bibitem[BS21]{JEP_2021__8__279_0}
\bysame, \emph{H\"older regularity for the spectrum of translation flows},
  Journal de l{\textquoteright}\'Ecole polytechnique {\textemdash}
  Math\'ematiques \textbf{8} (2021), 279--310 (en).

\bibitem[BS22]{bufetov2022substitution}
\bysame, \emph{On substitution automorphisms with pure singular spectrum},
  Mathematische Zeitschrift \textbf{301} (2022), no.~2, 1315--1331.

\bibitem[CDK17]{coven2017topological}
Ethan~M. Coven, Frederik~M. Dekking, and Michael~S. Keane, \emph{Topological
  conjugacy of constant length substitution dynamical systems}, Indagationes
  Mathematicae \textbf{28} (2017), no.~1, 91--107.

\bibitem[CQY16]{coven2016computing}
Ethan~M. Coven, Anthony Quas, and Reem Yassawi, \emph{Computing automorphism
  groups of shifts using atypical equivalence classes}, Discrete Analysis
  \textbf{2016} (2016), no.~3.

\bibitem[Dek14]{dekking2014structure}
Frederik~M. Dekking, \emph{On the structure of {T}hue--{M}orse subwords, with
  an application to dynamical systems}, Theoretical Computer Science
  \textbf{550} (2014), 107--112.

\bibitem[DG10]{drmota2010analysis}
Michael Drmota and Peter~J Grabner, \emph{Analysis of digital functions and
  applications}, Combinatorics, automata and number theory, Cambridge
  University Press, 2010, pp.~452--504.

\bibitem[DGM12]{dedecker2012almost}
Jerome Dedecker, S{\'e}bastien Gou{\"e}zel, and Florence Merlevede, \emph{The
  almost sure invariance principle for unbounded functions of expanding maps},
  ALEA: Latin American Journal of Probability and Mathematical Statistics
  \textbf{9} (2012), 141--163.

\bibitem[DHS99]{durand1999substitutional}
Fabien Durand, Bernard Host, and Christian Skau, \emph{Substitutional dynamical
  systems, bratteli diagrams and dimension groups}, Ergodic Theory and
  Dynamical Systems \textbf{19} (1999), no.~4, 953--993.

\bibitem[DL22]{durand2022decidability}
Fabien Durand and Julien Leroy, \emph{Decidability of the isomorphism and the
  factorization between minimal substitution subshifts}, Discrete Analysis
  (2022).

\bibitem[DT89]{DumontThomas}
Jean-Marie Dumont and Alain Thomas, \emph{Systemes de numeration et fonctions
  fractales relatifs aux substitutions}, Theoretical Computer Science
  \textbf{65} (1989), no.~2, 153 -- 169.

\bibitem[Dur00]{durand2000linearly}
Fabien Durand, \emph{Linearly recurrent subshifts have a finite number of
  non-periodic subshift factors}, Ergodic Theory and Dynamical Systems
  \textbf{20} (2000), no.~4, 1061--1078.

\bibitem[Eis13]{eisner2013linear}
Tanja Eisner, \emph{Linear sequences and weighted ergodic theorems}, Abstract
  and Applied Analysis, vol. 2013, Hindawi, 2013.

\bibitem[EP17]{emme2017asymptotic}
Jordan Emme and Alexander Prikhod'Ko, \emph{On the asymptotic behaviour of the
  correlation measure of sum-of-digits function in base 2}, Integers:
  Electronic Journal of Combinatorial Number Theory \textbf{17} (2017), A58.

\bibitem[For22]{forni2022twisted}
Giovanni Forni, \emph{Twisted translation flows and effective weak mixing},
  Journal of the European Mathematical Society \textbf{24} (2022), no.~12,
  4225--4276.

\bibitem[LM95]{lind_marcus_1995}
Douglas Lind and Brian Marcus, \emph{An introduction to symbolic dynamics and
  coding}, Cambridge University Press, 1995.

\bibitem[LQY17]{liu2017unbounded}
Qinghui Liu, Yanhui Qu, and Xiao Yao, \emph{Unbounded trace orbits of
  {Thue--Morse Hamiltonian}}, Journal of Statistical Physics \textbf{166}
  (2017), 1509--1557.

\bibitem[Ma{\~n}17]{manibo2017lyapunov}
Neil Ma{\~n}ibo, \emph{Lyapunov exponents for binary substitutions of constant
  length}, Journal of Mathematical Physics \textbf{58} (2017), no.~11.

\bibitem[MM20]{marshall2020modulus}
Juan Marshall-Maldonado, \emph{Modulus of continuity for spectral measures of
  suspension flows over {Salem} type substitutions}, arXiv preprint
  arXiv:2009.13607 (2020).

\bibitem[MR10]{mauduit2010probleme}
Christian Mauduit and Jo{\"e}l Rivat, \emph{Sur un probleme de {G}elfond: la
  somme des chiffres des nombres premiers}, Annals of Mathematics (2010),
  1591--1646.

\bibitem[Que10]{queffelec2010substitution}
Martine Queff{\'e}lec, \emph{Substitution dynamical systems-spectral analysis},
  vol. 1294, Springer, 2010.

\bibitem[Que18]{queffelec2018questions}
\bysame, \emph{Questions around the {T}hue--{M}orse sequence}, Unif. Distrib.
  Theory \textbf{13} (2018), no.~1, 1--25.

\bibitem[Tre20]{trevino2020quantitative}
Rodrigo Trevi{\~n}o, \emph{Quantitative weak mixing for random substitution
  tilings}, arXiv preprint arXiv:2006.16980 (2020).

\bibitem[Via14]{viana2014lectures}
Marcelo Viana, \emph{Lectures on {Lyapunov} exponents}, vol. 145, Cambridge
  University Press, 2014.

\end{thebibliography}
\end{document}